\documentclass[a4paper, 11pt]{article}

\usepackage{mathtools, amssymb, amsthm, bbm}
\usepackage{graphicx}
\usepackage{booktabs, multirow}
\usepackage{url}
\usepackage[colorlinks,citecolor=blue]{hyperref}
\usepackage[round]{natbib}
\usepackage{authblk}
\usepackage[margin = 1.15in]{geometry}
\usepackage{makecell}

\theoremstyle{plain} 
\newtheorem{thm}{Theorem}[section] 
\newtheorem{lem}[thm]{Lemma} 
\newtheorem{prop}[thm]{Proposition}

\theoremstyle{definition}

\theoremstyle{remark}

\newcommand{\one}{\mathbbm{1}}
\DeclareMathOperator*{\argmin}{argmin} 
\newcommand{\icx}{\preceq_{\mathrm{icx}}}
\newcommand{\icv}{\preceq_{\mathrm{icv}}}

\begin{document}

\title{Consistent estimation of distribution functions \\ under increasing concave and convex stochastic ordering}
\author{Alexander Henzi} 
\affil{University of Bern, Switzerland \\ \texttt{alexander.henzi@stat.unibe.ch}}
\date{\today} 
\maketitle

\begin{abstract}
A random variable $Y_1$ is said to be smaller than $Y_2$ in the increasing concave stochastic order if $\mathbb{E}[\phi(Y_1)] \leq \mathbb{E}[\phi(Y_2)]$ for all increasing concave functions $\phi$ for which the expected values exist, and smaller than $Y_2$ in the increasing convex order if $\mathbb{E}[\psi(Y_1)] \leq \mathbb{E}[\psi(Y_2)]$ for all increasing convex $\psi$. This article develops nonparametric estimators for the conditional cumulative distribution functions $F_x(y) = \mathbb{P}(Y \leq y \mid X = x)$ of a response variable $Y$ given a covariate $X$, solely under the assumption that the conditional distributions are increasing in $x$ in the increasing concave or increasing convex order. Uniform consistency and rates of convergence are established both for the $K$-sample case $X \in \{1, \dots, K\}$ and for continuously distributed $X$.
\end{abstract}

\section{Introduction}
The nonparametric estimation of distribution functions under stochastic order restrictions is a classical problem in statistics. It can be formulated very generally as the task to estimate the conditional distributions of a random variable $Y$ given a covariate $X$, solely under the assumption that these distributions are increasing in a certain stochastic order. The classical and best understood order is first order stochastic dominance, requiring that the conditional cumulative distribution functions (CDFs) $F_x(y) = \mathbb{P}(Y \leq y \mid X = x)$ are decreasing in $x$ for every fixed $y \in \mathbb{R}$. \citet{Brunk1966} were the first to consider this constrained estimation problem in the two sample case $X \in \{1,2\}$. Almost 40 years later, \citet{Elbarmi2005} have described an estimator for $X$ taking $K$ discrete ordered values, say, $X \in \{1, \dots, K\}$, and again after more than a decade, \citet{Moesching2020} extended it to continuously distributed $X$. In a further leap of complexity, \citet{Henzi2021c} have shown that consistent estimation under first order stochastic dominance is even possible with partially ordered covariates $X \in \mathbb{R}^d$. Stronger orders considered in the literature are the uniform stochastic ordering and the likelihood ratio order, see \citet{Elbarmi2016} and \citet{Moesching2020a} and the references therein. A weaker constraint is stochastic precedence \citep{Arcones2002}, and a structurally different stochastic order is the peakedness order, where the variability of the conditional distributions of $Y$ around a center is increasing in the covariate \citep{Rojo2007, Elbarmi2012, Elbarmi2017}. A common attractive feature of all these order restricted estimators is that they do not require the specification of tuning parameters, and automatically adapt to the smoothness of the underlying functions that are estimated, such as in \citet[Theorems 3.3 and 3.4]{Moesching2020}.

So far, the main efforts in developing estimators under stochastic order restrictions have been focused on first order stochastic dominance and stronger orders, and consistency results in the case of continuously distributed $X$ have only been derived for first order stochastic dominance. This is a limitation insofar as these orders require the conditional CDFs $F_x(y)$ to be decreasing in $x$ for all fixed $y$. In particular, the CDFs for different values of $x$ are not allowed to cross, which in practice often happens in the tails when the variability of $Y$ increases (or decreases) with $x$. The purpose of this article is to develop consistent estimators under the increasing concave and increasing convex stochastic order, which are weaker orders applicable in situations where first order stochastic dominance is not appropriate. Estimation under the increasing concave order has been studied before by \citet{Rojo2003} and \citet{Elbarmi2009} in the two-sample case $X \in \{1, 2\}$. In this article, an estimator generalizing the one by \citet{Elbarmi2009} to the $K$-sample case and continuously distributed $X$ is proposed, and uniform consistency and rates of convergence are established.

For two random variables $Y_1$ and $Y_2$ with finite expected values, $Y_1$ is said to be smaller in the increasing concave order than $Y_2$ if $\mathbb{E}[\phi(Y_1)] \leq \mathbb{E}[\phi(Y_2)]$ for all increasing concave functions $\phi$ for which the expectations exist. Similarly, $Y_1$ is smaller than $Y_2$ in the increasing convex order if $\mathbb{E}[\psi(Y_1)] \leq \mathbb{E}[\psi(Y_2)]$ for all increasing convex functions $\psi$, which is equivalent to $-Y_2$ being smaller than $-Y_1$ in the increasing concave order. In the following, these orders are abbreviated as  $Y_1  \icv Y_2$ and  $Y_1  \icx Y_2$, respectively, and $\icx$ and $\icv$ are both used as orders on random variables and on their CDFs. Another characterization \citep[see][Chapter 4]{Shaked2007} for the increasing concave order is
\[
\mathbb{E}[(y - Y_1)_+] = \int_{-\infty}^{y} F_1(t) \, dt \geq \int_{-\infty}^{y} F_2(t) \, dt = \mathbb{E}[(y - Y_2)_+], \ y \in \mathbb{R},
\]
where $(z)_+ = \max(z, 0)$, and $F_1$ and $F_2$ are the CDFs of $Y_1$ and $Y_2$, respectively. For the increasing convex order, the analogous condition reads as
\[
\mathbb{E}[(Y_1 - y)_+] = \int_y^{\infty} 1-F_1(t) \, dt \leq \int_y^{\infty} 1-F_2(t) \, dt = \mathbb{E}[(Y_2 - y)_+], \ y \in \mathbb{R}.
\]
A useful sufficient condition for the increasing concave order is that the CDFs $F_1$ and $F_2$ cross at a single point $y_0$ with $F_1(y) \leq F_2(y)$ for $y \leq y_0$ and $F_1(y) \geq F_2(y)$ for $y \geq y_0$, or with the reverse inequalities for the CDFs in case of the increasing convex order. The increasing concave order is well-known in economics as second order stochastic dominance, with ``second order'' referring to the fact that monotonicity is required for the integrated CDFs and not for the CDFs themselves. If $Y_1$ and $Y_2$ are portfolio returns, then $Y_1 \icv Y_2$ means that all individuals with increasing concave utility functions $\phi$, i.e.~all risk-averse utility maximizers, prefer $Y_2$ over $Y_1$. In the literature on finance and insurance, the increasing convex order appears under the name stop-loss order, a term introduced by \citet{Goovaerts1982} referring to the characterization $\mathbb{E}[(Y_1 - y)_+] \leq \mathbb{E}[(Y_2 - y)_+]$, which states that the expected stop-loss of $Y_2$ over any retention limit $y$ exceeds the stop-loss of $Y_1$. This suggests using $\icx$ as an order for comparing risks.

In the economics and finance literature, research has so far mainly been focused on developing tests for verifying if conditional distributions satisfy $\icv$- or $\icx$-order constraints, rather than using these orders as restrictions for estimating distributions. \citet{Baringhaus2009}, \citet{Berrendero2011}, and \citet{Donald2016} develop methods for the two-sample case; see the references in these articles for earlier works on the two-sample testing problem. The $K$-sample case has been addressed by \citet{Linton2005} and \citet{Zhang2015}. The test by \citet{Linton2005}, which covers first, second and higher order stochastic dominance, allows for dependency among the observations and also adjustment by covariates. \citet{Zhang2015} assume independent data and apply an isotonic regression estimator corresponding to the intermediate estimator $\tilde{M}_{x}(y)$ in Section \ref{sec:estimation} of this article. For a continuous covariate, \citet{Hsu2019} and \citet{Chetverikov2021} propose more general tests of monotonicity, which are also applicable to testing higher order stochastic dominance.

If one detaches the increasing concave (convex) order from its economic interpretation, then it can be viewed as an order relation where $Y$ increases with $X$, but its variability decreases (increases). More precisely, $Y_1 \icv Y_2$ if and only if there exists a variable $Z$ such that $Z \preceq_{\mathrm{cx}} Y_1$ and $Z \preceq_{\mathrm{st}} Y_2$, where $\preceq_{\mathrm{st}}$ and $\preceq_{\mathrm{cv}}$ denote first order stochastic dominance and the convex order, respectively. The latter is a variability order requiring that $\mathbb{E}[c(Z)] \leq \mathbb{E}[c(Y_2)]$ for all convex functions $c$ such that the expectations exist \citep[see][Chapter 3 and Theorem 4.A.6]{Shaked2007}. For the increasing convex order, one has $Y_1 \preceq_{\mathrm{cx}} Z$ instead of $Z \preceq_{\mathrm{cx}} Y_1$. With this broader perspective, one can find practical situations where the $\icv$- and $\icx$-order arise as natural constraints. For example, \citet{Vittorietti2021} perform a study in materials science, where intuition about physical properties suggests a model $\mathcal{N}(\mu_k, \sigma_k^2)$ with $\mu_1 \leq \dots \leq \mu_K$ and $\sigma_1^2 \geq \dots \geq \sigma_K^2$ for measurements of a certain outcome in $K$ different types of materials. This model, also studied by \citet{Shi1994}, satisfies the increasing concave order, and the methodology in this article provides an alternative nonparametric method for estimating the conditional distributions in their applications. Another application where the $\icv$- and $\icx$-order may be useful is the post-processing of point forecasts. Following \citet{Henzi2021c}, if $X$ is a point forecast for an outcome variable $Y$, say, an expert's inflation forecast, then one would expect that $Y$ attains higher values as $X$ increases. \citet{Henzi2021c} suggest to quantify the uncertainty of $X$ by estimating the conditional CDFs $F_x$ on training data under the assumption that they are increasing in first order stochastic dominance in $x$. This is plausible for many variables with right-skewed distribution, such as accumulated precipitation \citep{Henzi2021c} or patient length of stay \citep{Henzi2021a, Henzi2021b}. But the assumption may fail to hold in situations where the variability of $Y$ strongly changes with $X$. The case study in Section \ref{sec:application} presents an example with income expectations, where such an effect can be observed and the $\icv$-order provides a more plausible constraint than first order stochastic dominance.

\section{Estimation} \label{sec:estimation}
To avoid redundancy, only the estimation for the increasing concave order is presented here; the necessary adaptations for the increasing convex order are straightforward. Let $(X_1, Y_1), \dots, (X_n, Y_n) \in \mathbb{R} \times \mathbb{R}$ be covariate-observation pairs based on which the conditional distributions are to be estimated. In the literature on estimation under stochastic order restrictions, the CDFs $F_x(y)$ are often only estimated at the distinct values $x_1 < \dots < x_d$ of $X_1, \dots, X_n$ and $y_1 < \dots < y_m$ and of $Y_1, \dots, Y_n$, and frequently used estimation methods are nonparametric maximum likelihood estimation (NPLME) \citep[e.g.~in][]{Dykstra1991, Moesching2020a} and monotone least squares regression \citep[e.g.~in][]{Elbarmi2005, Moesching2020}. However, these two approaches turn out to be unrewarding in the case of the increasing concave order. Firstly, they lead to a constrained optimization problem with $\mathcal{O}(n^2)$ variables in general, namely the estimators $\hat{F}_{x_i}(y_j)$ for $F_{x_i}(y_j)$, which is not efficiently solvable for large $n$. And secondly, for the $\icv$-constrained estimator, proving consistency using the definition of the estimator as maximizer of the likelihood or least squares estimator seems intractable. The construction here is therefore an indirect approach. For $x, y \in \mathbb{R}$, define
\[
M_x(y) = \int_{-\infty}^y F_x(t) \, dt = \mathbb{E}[(y - Y)_+ \mid X = x].
\]
Under the assumption that $F_x \icv F_{x'}$ if $x \leq x'$, the quantities $M_x(y)$ should be decreasing in $x$ for all fixed $y$, and they satisfy
\[
M_x'(y+) = \lim_{h \rightarrow 0, \, h > 0} \frac{M_x(y + h) - M_x(y)}{h} = F_x(y).
\]
This suggests that an estimator $\hat{M}_x$ for $M_x$ may yield, under some conditions, an estimator for $F_x$. We restrict the estimation of $F_x(y)$ to $x \in \{x_1, \dots, x_d\}$; in Section \ref{sec:consistency}, it is shown that under a continuity assumption, any interpolation method to obtain estimates for $x \not\in\{x_1, \dots, x_d\}$ is sufficient for uniform consistency.

Since $M_{X_i}(y)$ equals the expected value $\mathbb{E}[(y - Y)_+ \mid X = X_i]$, a reasonable estimator for it is the antitonic least squares regression $\tilde{M}_{X_1}(y), \dots, \tilde{M}_{X_n}(y)$ of $(y - Y_1)_+, \dots, (y - Y_n)_+$ with covariates $X_1, \dots, X_n$, that is,
\[
[\tilde{M}_{X_1}(y), \dots, \tilde{M}_{X_n}(y)] \ = \argmin_{\eta \in \mathbb{R}^n: \, \eta_i \geq \eta_j \text{ if } X_i \leq X_j} \sum_{i = 1}^n [\eta_i - (y - Y_i)_+]^2.
\]
The order constraints enforce $\tilde{M}_{X_i}(y) = \tilde{M}_{X_j}(y)$ if $X_i = X_j$, so the above problem is equivalent to the reduced, weighted antitonic regression
\[
[\tilde{M}_{x_1}(y), \dots, \tilde{M}_{x_d}(y)] \ = \argmin_{\eta \in \mathbb{R}^d: \, \eta_1 \geq \dots \geq \eta_d}\sum_{i = 1}^d w_i [\eta_i - h_i(y)]^2,
\]
where $w_i = \#\{j \leq n: \, X_j = x_i\}$, $i = 1, \dots, d$, and
\[
h_i(y) = \frac{1}{w_i}\sum_{j: \, X_j = x_i} (y - Y_j)_+.
\]
This antitonic regression has the min-max-representation
\begin{equation} \label{eq:minmax}
	\tilde{M}_{x_i}(y) = \min_{k = 1, \dots, i} \max_{j = k, \dots, d}  \frac{1}{\sum_{s = k}^{j} w_s} \sum_{s = k}^j w_s h_s(y),
\end{equation}
see Equations (1.9)-(1.13) of \citet{Barlow1972}. In principle, one could now try to estimate $F_{x_i}(y)$ by taking the right-sided derivative of $\tilde{M}_{x_i}(\cdot)$ at $y$. However, $\tilde{M}_{x_i}$ is not necessarily convex and therefore its derivative may be decreasing and not a CDF. To correct this, let $\hat{M}_{x_i}$ be the greatest convex minorant to the function $y \mapsto \tilde{M}_{x_i}(y)$, which is the pointwise greatest convex function bounded by $\tilde{M}_{x_i}$ from above, and define $\hat{F}_{x_i}(y)$ as the right-hand slope of $\hat{M}_{x_i}(\cdot)$ at $y$. The following proposition, which is a consequence of the above min-max-formula and basic properties of greatest convex minorants, shows that this is a valid strategy. Its proof, and the proofs of all subsequent theoretical results, are deferred to the appendix.

\begin{prop} \label{prop:welldefined}
	\begin{itemize}
		\item[]
		\item[(i)] The functions $\tilde{M}_{x_i}(y)$ and $\hat{M}_{x_i}(y)$ are increasing and piecewise linear in $y$ for fixed $i \in \{1, \dots, d\}$, and decreasing in $i$ for fixed $y \in \mathbb{R}$.
		\item[(ii)] The functions $\hat{F}_{x_i}(y)$ for fixed $i \in \{1, \dots, d\}$ are piecewise constant CDFs with $F_{x_i}(y) = 0$ for $y < y_1$ and $F_{x_i}(y) = 1$ for $y \geq y_m$.
	\end{itemize}
\end{prop}

In practice, it is not possible to compute $\tilde{M}_{x_i}(y)$ and $\hat{M}_{x_i}(y)$ at all $y \in \mathbb{R}$. Although these functions are piecewise linear, there is no efficient procedure to identify the knots where their slope changes. A pragmatic solution is to evaluate $\tilde{M}_x$ and $\hat{M}_x$ on a fine grid $t_1 < \dots < t_k$ with $t_1 = y_1$ and $t_k = y_m$, and interpolate linearly in between. This has the consequence that the CDFs $\hat{F}_{x_i}$ can only put mass on $t_1, \dots, t_k$. By a standard result about isotonic regression (see Appendix \ref{app:gcm}), the right-sided slope of the greatest convex minorant to the interpolation of $(t_1, \tilde{M}_{x_i}(t_1)), \dots, (t_k, \tilde{M}_{x_i}(t_k))$ equals the isotonic regression of the slopes
\[
\tilde{F}_{x_i}(t_j) = \frac{\tilde{M}_{x_i}(t_{j + 1}) - \tilde{M}_{x_i}(t_j)}{t_{j + 1} - t_j}
\]
with weights $t_{j + 1} - t_j$, $j = 1, \dots, k - 1$. This isotonic regression directly yields the estimators for the conditional CDFs,
\[
[\hat{F}_{x_i}(t_1), \dots, \hat{F}_{x_i}(t_{k - 1})\big] \ = \argmin_{\xi \in \mathbb{R}^{k-1}: \, \xi_1 \leq \dots \leq \xi_{k-1}} \sum_{j = 1}^{k-1} (t_{j + 1} - t_{j})[\xi_j - \tilde{F}_{x_i}(y_j)]^2,
\]
and $\hat{F}_{x_i}(t_k) = 1$ by Proposition \ref{prop:welldefined} (ii) if $t_k = y_m$. To summarize, the estimation procedure consists of two series of monotone regressions, informally speaking one in the $X$-direction for fixed threshold $y$ to obtain $\icv$-ordered distributions, and another in the $Y$-direction for fixed covariate $x_i$ to ensure that the CDFs are increasing. It is not necessary to compute the functions $\hat{M}_{x_i}$ explicitly, since the computation of the greatest convex minorant is indirect via its right-hand slope. The exact solution of monotone regression problems can be obtained efficiently with the Pool-Adjacent Violators Algorithm (PAVA), which has complexity $\mathcal{O}(N)$ with sorted covariate and sample size $N$. Hence the overall complexity of the estimation procedure is $\mathcal{O}(n^2)$ if the number of distinct values in $X_1, \dots, X_n$ or $Y_1, \dots, Y_n$ grows at the rate $\mathcal{O}(n)$.

If the distinct values $y_1, \dots, y_m$ of $Y_1, \dots, Y_n$ are taken as the grid for computation, then the estimated distributions $\hat{F}_{x_i}$ can only put mass on the actual observations in the data, and they are equal to the conditional empirical cumulative distribution functions (ECDF) if these already satisfy the increasing concave order condition. That is, if $\check{F}_{x_j}$ is the ECDF of all $Y_i$ with $X_i = x_j$ and if $\check{F}_{x_i} \icv \dots \icv \check{F}_{x_d}$, then $\hat{F}_{x_j} = \check{F}_{x_j}$ for $j = 1, \dots, d$. If in addition $X_1, \dots, X_n$ are pairwise distinct, $\hat{F}_{X_i}$ is the Dirac measure at $Y_i$ for $i = 1, \dots, n$. The estimators under first order stochastic dominance \citep{Elbarmi2005, Moesching2020} also have this property. However, with the increasing concave order, even if the grid $\{t_1, \dots, t_k\}$ contains $\{y_1, \dots, y_m\}$, the estimator can put probability mass on points outside of $\{y_1, \dots, y_m\}$. In particular, if the response variable is known to take values in a discrete set, say $\mathbb{Z}$, then the grid should be chosen within this set to avoid positive estimated probabilities outside of the actual support.

The increasing concave order is preserved under pointwise convex combinations of CDFs, i.e.~if $F_1 \icv F_2$ and $G_1 \icv G_2$, then also $\lambda F_1 + (1-\lambda) G_1 \icv \lambda F_2 + (1-\lambda)G_2$ for $\lambda \in (0,1)$. This fact opens the possibility to combine the estimation procedure with sample splitting as suggested in \citet{Henzi2021c} for first order stochastic dominance. Instead of estimating the conditional distributions with the complete dataset, one may draw random subsamples from the data and aggregate the estimated conditional CDFs from each run by their pointwise average. This subsample aggregation (subagging) yields smoother estimated CDFs and prevents overfitting. Alternatively, the data can also be partitioned into several disjoint subsets instead of drawing subsamples, and \citet{Banerjee2019} have proved that this divide and conquer strategy may lead to better convergence rates of isotonic mean regression. Partitioning of the data is a valid strategy for very large datasets, but in smaller datasets it is more desirable to apply subagging to avoid that the estimator depends too strongly on the chosen partition. In principle, the averaging step in subagging or sample splitting could also be done on the level of the estimators $\tilde{M}_x$ instead of the CDFs $\hat{F}_x$, but if the goal is to obtain smoother CDFs, it is more natural to average the $\hat{F}_x$.

Finally, we show that the estimator proposed here generalizes the one by \citet{Elbarmi2009} for the case $X_i \in \{1, 2\}$. Their estimator depends on a parameter $\alpha \in [0,1]$, and equality holds for $\alpha = \#\{i \leq n: \, X_i = 1\} / n$. This follows from the fact that with $\check{M}_{j}(y) = \int_{-\infty}^y \check{F}_j(t) \, dt$, one can write $\tilde{M}_j(y)$ as
\[
\tilde{M}_{j}(y) = \one\{\check{M}_j(y) \geq \check{M}_2(y)\}\check{M}_1(y) + \one\{\check{M}_1(y) < \check{M}_2(y)\}[\alpha\check{M}_1(y) + (1-\alpha)\check{M}_1(y)],
\]
for $j = 1, 2$, where $\one$ is the indicator function. Taking the right-hand slope of the greatest convex minorant of the above functions then yields Equation (8) from \citet{Elbarmi2009}. The choice $\alpha = \#\{i \leq n: \, X_i = 1\} / n$ was already suggested in their article, and it corresponds to the natural weight for which $\tilde{M}_j$, $j = 1, 2$, are the antitonic regression estimators.

\section{Uniform consistency} \label{sec:consistency}
The following notation and assumptions are required for stating the theorems about uniform consistency. Let $(X_{ni},Y_{ni})$, $i = 1, \dots, n$, $n \in \mathbb{N}$, be a triangular array defined on a measurable space $(\Omega, \mathcal{F})$ with a probability measure $\mathbb{P}$. For a sequence of events $(A_n)_{n \in \mathbb{N}} \subset \mathcal{F}$, the statement ``$A_n$ holds with asymptotic probability one'' means $\lim_{n \rightarrow \infty} \mathbb{P}(A_n) = 1$. The covariates $X_{n1}, \dots, X_{nn}$ are assumed to be independent and have distinct values $x_1 < \dots < x_d$, and the response variables $Y_{n1}, \dots, Y_{nn}$ are independent conditional on $X_{n1}, \dots, X_{nn}$ such that $\mathbb{P}(Y_{ni} \leq y \mid X_{ni}) = F_{X_{ni}}$, with the CDFs $F_{x}$ increasing in $x$ in the increasing concave order. The distinct values of $Y_{n1}, \dots, Y_{nn}$ are again denoted by $y_1 < \dots < y_m$. A subscript $n$ in $\tilde{M}_{n; x}(y)$, $\hat{M}_{n; x}(y)$, and $\hat{F}_{n; x}(y)$ will be used to indicate that these quantities depend on the sample size $n$, but the dependency of $m$ and $d$ on $n$ is not written explicitly to lighten the notation. If $x \not\in \{x_1, \dots, x_d\}$, it is only assumed that $\tilde{M}_{n;x_i}(y) \geq \tilde{M}_{n;x}(y) \geq \tilde{M}_{n;x_{i + 1}}(y)$ for all $y \in \mathbb{R}$ if $x \in [x_i, x_{i + 1})$, and $\tilde{M}_{n;x}(y) = \tilde{M}_{n;x_1}(y)$ if $x < x_1$ or $\tilde{M}_{n;x}(y) = \tilde{M}_{n;x_d}(y)$ if $x \geq x_d$. The same property then also holds for $\hat{M}_{n;x}$. 

The key condition for proving consistency of $\hat{F}_{n;x}(y)$ is the following.

\medskip
\noindent\textbf{(A)} \ There exists $(c_n)_{n \in \mathbb{N}} \subset [0, \infty)$ and a sequence of sets $(I_n)_{n \in \mathbb{N}}$, $I_n \subset \mathbb{R}$, such that
\[
\lim_{n \rightarrow \infty} \mathbb{P}\left(\sup_{y \in \mathbb{R}, \, x \in I_n} |\tilde{M}_{n;x}(y) - M_{x}(y)| \geq c_n \right) = 0.
\]

Sufficient conditions for (A) will be given below, and the convergence rate $c_n$ depends on whether $X$ is discrete or continuously distributed and on the tail properties of $F_x$. If $X_{n1}, \dots, X_{nn} \in \{1, \dots, K\}$, one can simply set $I_n = \{1, \dots, K\}$. For continuously distributed covariates on an interval $I$, $I_n$ will be of the form $I_n = \{x \in I: \, x \pm \delta_n \in I\}$ with $\delta_n \rightarrow 0$, that is, it is in general not possible to show consistency at the boundary of the covariate domain. This is also the case in isotonic mean regression and estimation under first order stochastic dominance \citep{Moesching2020}.

The following proposition establishes the connection between the uniform consistency of $\tilde{M}_{n;x}(y)$ and $\hat{F}_{n;x}(y)$.

\begin{prop} \label{prop:consistency}
	Assume that (A) holds and $I \subseteq \mathbb{R}$ is a set such that $I_n \subseteq I$, $n \in \mathbb{N}$.
	\begin{itemize}
		\item[(i)] If there exist $J \subseteq \mathbb{R}$ and constants $C\geq 0$, $\beta > 0$ such that $|F_x(y) - F_x(z)| \leq C|y - z|^{\beta}$ for all $y, z \in J$, $x \in I$, then with $J_n = \{y \in J: \, y \pm c_n^{1/(1 + \beta)} \in J\}$,
		\[
		\lim_{n \rightarrow \infty} \mathbb{P}\left(\sup_{y \in J_n, \, x \in I_n} |\hat{F}_{n;x}(y) - F_{x}(y)| \geq (2 + C)c_n^{\beta/(1+\beta)} \right) = 0.
		\]
		\item[(ii)] If the distribution functions $F_x$, $x \in I$, have support in $\mathbb{Z}$ and if $\tilde{M}_{n; x}$ is computed with grid $\{y_1, y_1 + 1, \dots, y_{m} - 1, y_m\}$, then 
		\[
		\lim_{n \rightarrow \infty} \mathbb{P}\left(\sup_{y \in \mathbb{R}, \, x \in I_n} |\hat{F}_{n;x}(y) - F_{x}(y)| \geq 2c_n \right) = 0.
		\]
	\end{itemize}
\end{prop}
Proposition \ref{prop:consistency} shows that if $\tilde{M}_{n;x}(y)$ is uniformly consistent in $x$ and $y$ at a rate $c_n$, then the estimator $\hat{F}_{n;x}(y)$ is also uniformly consistent. When the response variable is integer-valued, $\hat{F}_{n;x}$ is consistent at the same rate. Otherwise, if the distribution functions $F_x$ are H\"older continuous with index $\beta$, the corresponding rate for $\hat{F}_{n;x}(y)$ is $c_n^{\beta/(1+\beta)}$, for example $c_n^{1/2}$ if the $F_x$ are Lipschitz continuous. Note that in the case $J = \mathbb{R}$,  the sets $J_n$ in Proposition \ref{prop:consistency} (i) are also equal to $\mathbb{R}$, and in (ii), the support of $F_x$, $x \in I$, could be any discrete lattice instead of $\mathbb{Z}$.

We proceed to state conditions under which (A) holds. For the $K$-sample case, the assumption on the covariate is the following.

\medskip
\noindent \textbf{(K)} \ The covariates take values in $I = \{1, \dots, K\}$, and $\min_{j = 1, \dots, K} \mathbb{P}(X_{ni} = j) = p$ for some $p > 0$.

\medskip
Instead of $\{1, \dots, K\}$ in (K), the set $I$ could be any discrete ordered set with cardinality $K$. In the continuous case, the assumptions are analogous to (A.1) and (A.2) in \citet{Moesching2020}.

\medskip
\noindent \textbf{(C1)} \ The covariates $X_{n1}, \dots, X_{nn}$ admit a Lebesgue density bounded away from zero by $p > 0$ on an interval $I$.

\medskip
\noindent \textbf{(C2)} \ There exists a constant $L > 0$ such that for all $u, v \in I$ and $y \in \mathbb{R}$,
\[
|M_{u}(y) - M_{v}(y)| \leq L|u - v|.
\]

Note that the set $I$ and the constants $p$ in (K) and (C1) and $L$ in (C2) do not depend on $n$. Condition (C1) could be replaced by the weaker assumption that the number of points in every subinterval of $I$ of a certain size grows sufficiently fast, like in (A.2) of \citet[see also their Remark 3.2]{Moesching2020}. In particular, it is not necessary to assume that the covariates $X_{n1}, \dots, X_{nn}$ are pairwise distinct or independent. However, this more general condition would require to introduce additional notation and constants. Similarly, in (K), it is sufficient that each value $j \in \{1, \dots, K\}$ is attained at least $n\delta$ times with asymptotic probability one for some $\delta > 0$. The Lipschitz assumption (C2) in the continuous case is standard in isotonic regression \citep[see e.g.][]{Yang2019, Dai2020}, and it could be replaced by H\"older continuity with index $\alpha \in (0,1)$ at the cost of a slower convergence rate.

Since the goal is to prove consistency for an estimator of the expected values $\mathbb{E}[(y - Y)_+ \mid X = x]$, it is natural that some additional assumptions on the tail behaviour of the distributions $F_x$ are required. In the two cases below, the set $I$ is assumed to be the one from (K) or from (C1), (C2).

\medskip
\noindent \textbf{(P)} \ There exist $\lambda > 2$ and $y_0 \geq 0$ such that for all $y \geq y_0$ and $x \in I$,
\[
\mathbb{P}(|Y| \geq y \mid X = x) \leq y^{-\lambda}.
\]

\medskip
\noindent \textbf{(E)} \ There exist $\lambda > 0$ and $y_0 \geq 0$ such that for all $y \geq y_0$ and $x \in I$,
\[
\mathbb{P}(|Y| \geq y \mid X = x) \leq \exp(-\lambda y).
\]

\begin{thm} \label{thm:rate}
	Condition (A) holds with
	\[
	c_n = \begin{dcases}
		4p^{-1/2}n^{-1/2 + 1/\lambda}\log(n)^{1/2 + 1/\lambda}, & \ \text{under (K) and (P)}, \\
		8p^{-1/2}\lambda^{-1}n^{-1/2}\log(n)^{3/2}, & \ \text{under (K) and (E)}, \\
		[4p^{-1/2} + L] n^{-1/3 + 2/(3\lambda)}\log(n)^{1/3 + 2/(3\lambda)}, & \ \text{under (C1), (C2), and (P)}, \\
		(2/\lambda)^{2/3}[4p^{-1/2} + L]n^{-1/3}\log(n), & \ \text{under (C1), (C2), and (E)},
	\end{dcases}
	\]
	and
	\[
	I_n = \begin{dcases}
		\{1, \dots, K\}, & \ \text{under (K)}, \\
		\{x \in I: \, x \pm n^{-1/3 + 2/(3\lambda)}\log(n)^{1/3 + 2/(3\lambda)} \in I\}, & \ \text{under (C1), (C2), and (P)}, \\
		\{x \in I: \, x \pm (2/\lambda)^{2/3}n^{-1/3}\log(n) \in I\}, & \ \text{under (C1), (C2), and (E)}.
	\end{dcases}
	\]
\end{thm}

In the $K$-sample case, Theorem \ref{thm:rate} implies uniform consistency at a rate of at least $(\log(n)/n)^{1/4}$ if the distribution functions $F_x(y)$ are Lipschitz continuous in $y$ and have exponential tails. This is slower than the $n^{-1/2}$-rate of the empirical distribution functions stratified by the $K$ covariate values, and suggests that this lower bound is not always tight. Indeed, if the conditional CDFs $F_j$, $j = 1, \dots, K$, have support on disjoint, pointwise increasing intervals, then $\hat{F}_{n;j}$ are equal to the ECDFs of the corresponding subsamples and hence known to converge at the faster rate. Nevertheless, the result extends the ones from the current literature. In the two-sample case $K = 2$, \citet{Rojo2003} establish strong uniform convergence and pointwise but not uniform root-$n$ convergence for their estimator, while \citet{Elbarmi2009} only prove strong uniform consistency, but do not derive rates of convergence.

For a continuously distributed covariate and exponential tails, $\tilde{M}_x(y)$ converges uniformly in $x$ and $y$ at a rate of $n^{-1/3}$ up to a logarithmic factor, which is known to be the global rate of convergence of the isotonic regression estimator. When the conditional distributions have power tails with exponent $\lambda$, the rate becomes slower by a factor of $n^{2/(3\lambda)}$. In general, the global $n^{-1/3}$-rate of convergence for isotonic regression does not require the assumption of exponential tails, but the results across the literature are not directly comparable. For example, \citet{Zhang2002} shows that with bounded second moments, the risk of the isotonic mean regression estimator, i.e.~the root mean squared error at the design points, scales at a rate of $n^{-1/3}$, whereas \citet{Yang2019} prove uniform consistency with the same rate (up to logarithmic factors) in the supremum norm under sub-gaussianity of the error terms. Theorem \ref{thm:rate} yields a stronger statement since convergence is also uniform in the parameter $y$, and with exponential tails it still matches the optimal global rate up to the logarithmic factor. For $\hat{F}_{x}(y)$, Theorem \ref{thm:rate} implies a rate of at least $n^{-1/6}$ under the favorable assumption (E) and Lipschitz continuity of $F_x(y)$ in $y$.

Isotonic regression for the mean has the property that it automatically adapts to the local smoothness of the underlying function; see for example \citet{Guntuboyina2018, Yang2019}. A slight adaptation in the proof of Theorem \ref{thm:rate} shows that this is also true for the estimator $\tilde{M}_x(y)$. For example, if $M_x(y)$ is constant in $x \in I$, for all fixed $y$, then the bound on the error under (E) becomes $8p^{-1/2}\lambda^{-1}\log(n)^{2} n^{-1/2}$ and is thus of order $n^{-1/2}$ up to the logarithmic factor. Similarly, rates in between $n^{-1/3}$ and $n^{-1/2}$ can be obtained when $|M_x(y) - M_{x'}(y)| \leq L|x-x'|^{v}$ for $x \in I$ and all $y$, with suitable $v > 1$. The same phenomenon occurs in part (i) of Proposition \ref{prop:consistency}, where faster convergence rates are obtained for larger exponents $\beta$ in the Hölder continuity assumption.

\section{Inference} \label{sec:inference}
Constructing confidence intervals for conditional distributions under stochastic order constraints is difficult. This section does not provide a solution to this problem for the $\icv$- and $\icx$-order, but it gives an overview of potential approaches and their pitfalls.

The limiting distributions of shape-constrained estimators are often rather complicated and depend on unknown characteristics of the underlying functions \citep{Park2012, Guntuboyina2018}, which makes it difficult to derive confidence intervals from them. For $\icv$-order in the $K$-sample case, \citet{Zhang2015} have shown  $n^{1/2}[M_{x_i}-\tilde{M}_{x_i}]$ converges to a process which is described by a min-max formula similar to \eqref{eq:minmax} applied to integrated Brownian bridges with rescaled time. To construct confidence intervals for $M_{x_i}(y)$ or $F_{x_i}(y)$ from these distributions, one would need knowledge about the true CDFs and about the groups where the constraints are binding, i.e.~hold with equality, which is usually not available in practice.

A more promising approach seem to be bootstrap methods, which have been explored by \citet{Park2012} in the case of first order stochastic dominance. However, the bootstrap is delicate in shape restricted regression, where it was shown that some popular bootstrap methods can be inconsistent \citep[][Section 4.2.1, and the references therein]{Guntuboyina2018}, and so far there has been no thorough analysis of this phenomenon in the case of estimating conditional distributions. For the maximum likelihood estimator of a decreasing density, the Grenander estimator, it has been shown that suitable smoothing can yield a consistent bootstrap estimator \citep{Groeneboom2010}. This suggests to study the smoothed estimator
\[
\hat{K}_{x;h}(y) = \int_{-\infty}^{\infty} K\left(\frac{y-t}{h}\right) \, d\hat{F}_{x}(t),
\]
where $K$ is a smooth CDF, $h > 0$ a bandwidth, and $\hat{F}_x$ an estimator under stochastic order restrictions. This modification, which has not been studied for smoothing restricted estimators of conditional distributions so far, may also be interesting from a pure estimation perspective as an alternative to the bagging procedure suggested in Section \ref{sec:estimation}.

Finally, \citet{Yang2019} developed a method for constructing simultaneous confidence intervals in isotonic mean regression which requires no tuning parameters. These confidence bands only assume subgaussian error distributions, and are similar to the restricted confidence intervals by \citet[Section 5]{Elbarmi2005} under first order stochastic dominance. In the case of the increasing concave order, they might help to construct confidence bands for $M_x(y)$ for varying $x$ and fixed $y$, but it is not obvious whether this allows to derive practically useful bands for the derivative $F_x(y) = M_x'(y)$.

\section{Simulations} \label{sec:simulations}
In the following simulation examples, the $\icv$- and $\icx$-order constrained estimators are compared to competitors in terms of the $L_1$ distance between the estimated and the true CDFs, and in terms of the mean absolute error (MAE) of quantile estimates, 
\[
\mathrm{L}_1(\hat{F}_n, F) = \mathbb{E}\left( \int_{-\infty}^{\infty} |\hat{F}_n(y) - F(y)| \, dy\right), \quad
dq_{\gamma}(\hat{F}_n, F) = \mathbb{E}\Big(|\hat{F}_n^{-1}(\gamma) - F^{-1}(\gamma)|\Big),
\]
where $\hat{F}_n$ denotes an estimator for $F$ and the expected value is taken over the sampling distribution of $\hat{F}_n$ for the given sample size. The expected value in the definition of $\mathrm{L}_1$ and $dq_{\gamma}$ is approximated by the empirical mean over $10'000$ simulations for the examples with discrete and $5'000$ simulations for those with continuously distributed covariates.

With covariate values $X \in [1, 4]$, the following three settings are considered:
\begin{align}
	Y_1 & = X^{1/2} + \big[1 + (X - 2) / (1 + (X - 2)^{2})^{1/2}\big] \varepsilon, \ \varepsilon \sim \text{Student}(\text{df} = 10), \label{eq:sim_student} \\
	Y_2 & \sim \text{Gamma}(\text{shape} = X, \, \text{rate} = X^{9/10}), \label{eq:sim_gamma} \\
	Y_3 & \sim \text{Beta-binomial}(n = 50, \, \alpha = X^3, \, \beta = 1 + X^3)	. \label{eq:sim_beta}
\end{align}
The conditional distributions of $Y_1$ given $X$ are ordered in the increasing convex order, and the those of $Y_2$ and $Y_3$ with respect to the increasing concave order; see Figure \ref{fig:sim_illustration} (a) for an illustration. In the $K$-sample case, $X$ takes values in $\{1, 4\}$, $\{1, 2, 3, 4\}$, and $\{1, 1.5, \dots, 3.5, 4\}$, i.e.~$K = 2, \, 4, \, 7$, which allows comparing the change in estimation error at previously available values of $X$ when the number of samples increases. For simulation examples with continuous covariate, the sample of $X$ is generated independent and uniformly distributed on $[1, 4]$. In all simulations the distinct observed values of the response variable are taken as grid for the computation of the $\icv$- and $\icx$-order constrained estimators.

\begin{table}
	\caption{Relative improvement in mean $L_1$ distance, mean absolute error of quantile estimates of $\icv$- and $\icx$-order constrained estimator compared to ECDF stratified by the value of $X$, for $K = 2, \, 4, \, 7$ and group sizes of $n = 30, \, 50$. \label{tab:ksample}}
	\bigskip
	\resizebox{\columnwidth}{!}{%
		\begin{tabular}{rr|*{4}{r}|*{4}{r}|*{4}{r}}
			\toprule
			\multicolumn{2}{c|}{$n = 30$} & \multicolumn{4}{c|}{Student \eqref{eq:sim_student}} & \multicolumn{4}{c|}{Gamma \eqref{eq:sim_gamma}} & \multicolumn{4}{c}{Beta-binomial \eqref{eq:sim_beta}} \\
			$K$ & $X$ & $\mathrm{L_1}$ & $dq_{0.1}$ & $dq_{0.5}$ & $dq_{0.9}$ & $\mathrm{L_1}$ & $dq_{0.1}$ & $dq_{0.5}$ & $dq_{0.9}$ & $\mathrm{L_1}$ & $dq_{0.1}$ & $dq_{0.5}$ & $dq_{0.9}$ \\
			\midrule
			2 & 1.0 & 0.00 & 0.00 & 0.00 & 0.00 & 0.02 & 0.00 & 0.00 & 0.07 & 0.00 & 0.00 & 0.00 & 0.00\\
			& 4.0 & 0.00 & 0.00 & 0.00 & 0.00 & -0.03 & 0.00 & 0.00 & -0.04 & 0.00 & 0.00 & 0.00 & 0.00\\
			\midrule
			4 & 1.0 & 0.00 & 0.00 & 0.00 & 0.00 & 0.06 & 0.00 & 0.05 & 0.12 & 0.00 & 0.00 & 0.00 & 0.00\\
			& 2.0 & -0.01 & -0.10 & 0.01 & -0.01 & 0.08 & 0.05 & 0.12 & 0.19 & 0.05 & 0.00 & 0.05 & 0.09\\
			& 3.0 & 0.06 & 0.19 & 0.08 & 0.05 & 0.09 & 0.13 & 0.18 & 0.21 & 0.13 & 0.09 & 0.09 & 0.12\\
			& 4.0 & 0.05 & 0.16 & 0.10 & 0.10 & 0.03 & 0.11 & 0.11 & 0.14 & 0.08 & 0.09 & 0.15 & 0.10\\
			\midrule
			7 & 1.0 & 0.00 & -0.01 & 0.00 & 0.00 & 0.08 & 0.04 & 0.10 & 0.13 & 0.01 & 0.00 & 0.00 & 0.01\\
			& 1.5 & -0.02 & -0.10 & 0.00 & -0.01 & 0.14 & 0.13 & 0.21 & 0.29 & 0.05 & 0.00 & 0.04 & 0.05\\
			& 2.0 & 0.01 & 0.01 & 0.05 & 0.00 & 0.16 & 0.20 & 0.26 & 0.35 & 0.12 & 0.05 & 0.12 & 0.14\\
			& 2.5 & 0.08 & 0.23 & 0.15 & 0.08 & 0.18 & 0.25 & 0.31 & 0.37 & 0.21 & 0.14 & 0.22 & 0.24\\
			& 3.0 & 0.13 & 0.33 & 0.23 & 0.19 & 0.17 & 0.28 & 0.32 & 0.36 & 0.28 & 0.24 & 0.22 & 0.29\\
			& 3.5 & 0.14 & 0.33 & 0.26 & 0.25 & 0.15 & 0.28 & 0.30 & 0.32 & 0.27 & 0.29 & 0.29 & 0.20\\
			& 4.0 & 0.09 & 0.23 & 0.19 & 0.21 & 0.07 & 0.19 & 0.19 & 0.22 & 0.15 & 0.15 & 0.26 & 0.19\\
			\midrule
			\midrule
			\multicolumn{2}{c|}{$n = 50$} & \multicolumn{4}{c|}{Student \eqref{eq:sim_student}} & \multicolumn{4}{c|}{Gamma \eqref{eq:sim_gamma}} & \multicolumn{4}{c}{Beta-binomial \eqref{eq:sim_beta}} \\
			$K$ & $X$ & $\mathrm{L_1}$ & $dq_{0.1}$ & $dq_{0.5}$ & $dq_{0.9}$ & $\mathrm{L_1}$ & $dq_{0.1}$ & $dq_{0.5}$ & $dq_{0.9}$ & $\mathrm{L_1}$ & $dq_{0.1}$ & $dq_{0.5}$ & $dq_{0.9}$ \\
			\midrule
			2 & 1.0 & 0.00 & 0.00 & 0.00 & 0.00 & 0.01 & 0.00 & 0.00 & 0.05 & 0.00 & 0.00 & 0.00 & 0.00\\
			& 4.0 & 0.00 & 0.00 & 0.00 & 0.00 & -0.02 & 0.00 & 0.00 & -0.02 & 0.00 & 0.00 & 0.00 & 0.00\\
			\midrule
			4 & 1.0 & 0.00 & 0.00 & 0.00 & 0.00 & 0.04 & 0.00 & 0.02 & 0.10 & 0.00 & 0.00 & 0.00 & 0.00\\
			& 2.0 & -0.01 & -0.07 & 0.00 & 0.00 & 0.05 & 0.02 & 0.08 & 0.14 & 0.03 & 0.00 & 0.02 & 0.08\\
			& 3.0 & 0.04 & 0.13 & 0.05 & 0.04 & 0.07 & 0.08 & 0.14 & 0.20 & 0.09 & 0.05 & 0.04 & 0.07\\
			& 4.0 & 0.04 & 0.13 & 0.07 & 0.07 & 0.03 & 0.07 & 0.09 & 0.11 & 0.06 & 0.06 & 0.12 & 0.05\\
			\midrule
			7 & 1.0 & 0.00 & 0.00 & 0.00 & 0.00 & 0.07 & 0.02 & 0.07 & 0.13 & 0.00 & 0.00 & 0.00 & 0.00\\
			& 1.5 & -0.01 & -0.07 & 0.00 & 0.00 & 0.11 & 0.09 & 0.16 & 0.25 & 0.03 & 0.00 & 0.02 & 0.03\\
			& 2.0 & 0.00 & -0.04 & 0.02 & -0.01 & 0.13 & 0.14 & 0.21 & 0.29 & 0.08 & 0.02 & 0.08 & 0.12\\
			& 2.5 & 0.05 & 0.15 & 0.09 & 0.03 & 0.15 & 0.19 & 0.26 & 0.32 & 0.16 & 0.09 & 0.19 & 0.20\\
			& 3.0 & 0.11 & 0.28 & 0.19 & 0.15 & 0.16 & 0.23 & 0.28 & 0.34 & 0.23 & 0.17 & 0.16 & 0.25\\
			& 3.5 & 0.13 & 0.27 & 0.23 & 0.22 & 0.14 & 0.24 & 0.27 & 0.31 & 0.25 & 0.26 & 0.26 & 0.16\\
			& 4.0 & 0.08 & 0.20 & 0.16 & 0.18 & 0.07 & 0.16 & 0.16 & 0.18 & 0.14 & 0.11 & 0.26 & 0.12\\
			\bottomrule 
			
		\end{tabular}
	}%
\end{table}

Table \ref{tab:ksample} shows the performance order restricted estimators compared to the ECDF in the $K$-sample case, with fixed group sizes $n = 30, \, 50$ as in \citet{Elbarmi2009}. One would expect that the improvement of the restricted estimator over the ECDF is larger when more constraints are binding. That is, when the conditional CDFs $F_x$ (or the functions $M_x$) are close to each other for the different values of $x$, and when the sample size is small, then the restricted estimator may gain precision by pooling information across the $K$ groups. The case study confirms this intuition. For $K = 2$ (groups $X = 1$ and $X = 4$) the ECDF already satisfies the order constraints in most cases and has similar estimation error as the constrained estimator. However, as new groups are included with covariate value $X$ in between those of the previously present groups, the order restricted estimators benefit from the larger total sample size and achieve a lower estimation error both globally, i.e.~in $\mathrm{L}_1$-distance, and for most quantiles considered. This improvement is larger for $n = 30$ than for $n = 50$, since with the larger sample size the ECDFs are already closer to the true distributions and require fewer corrections to satisfy the order constraints.

\begin{figure}[h]
	\centering
	\caption{(a) Deciles of the conditional distributions in the simulation examples \eqref{eq:sim_student}, \eqref{eq:sim_gamma}, \eqref{eq:sim_beta}. The median is depicted as a dashed line. (b) Quantile curves (levels $0.1, 0.3, 0.5, 0.7, 0.9$) for simulation example \eqref{eq:sim_gamma} together with $\icv$-ordered estimator ($n = 500$; ICV and subagging variant $\mathrm{ICV}_{\mathrm{sbg}}$ with $50$ subsamples of size $250 = n/2$). \label{fig:sim_illustration}}
	\bigskip
	\includegraphics[width = 0.9\textwidth]{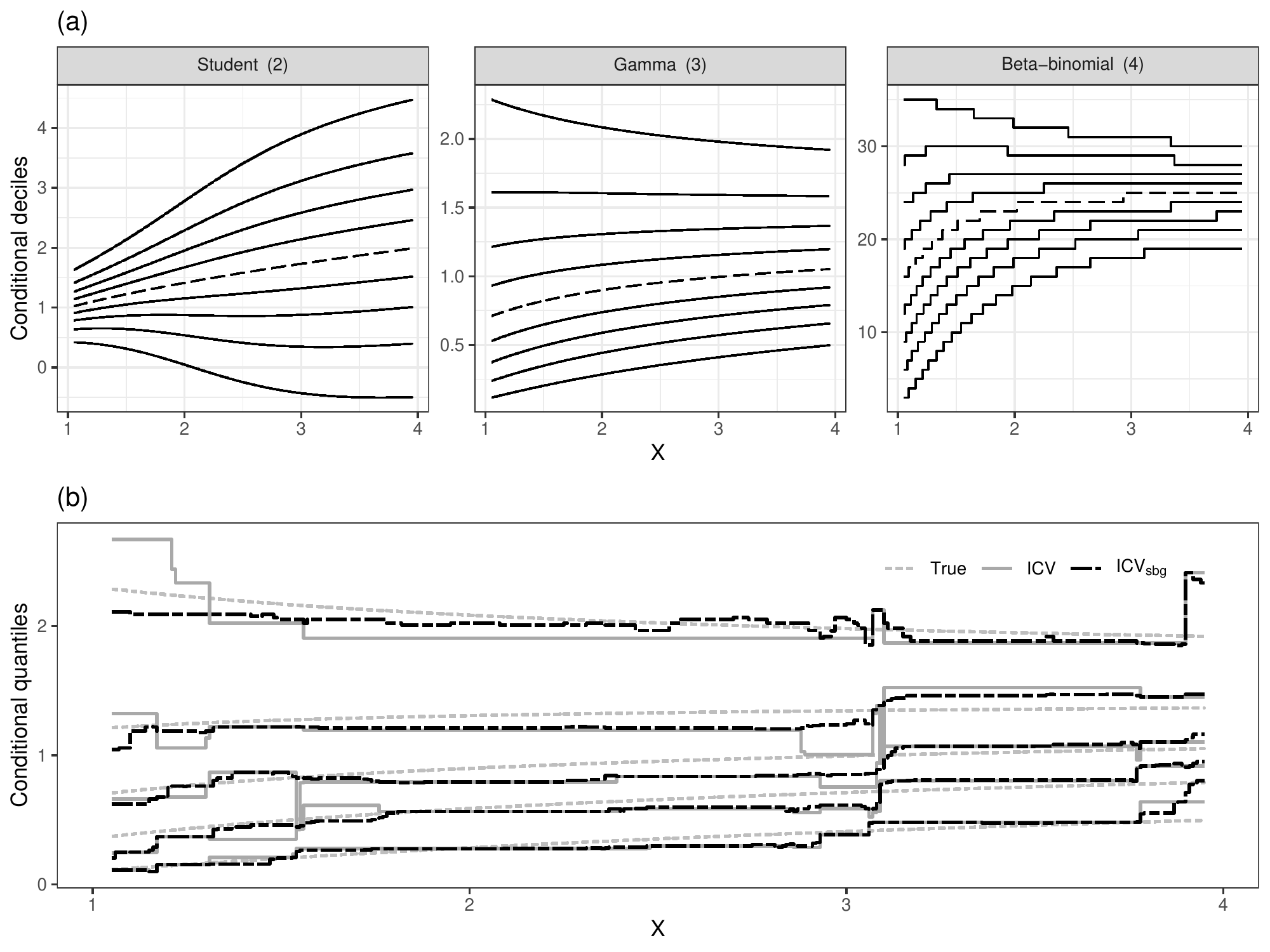}
\end{figure}

In the continuous case, the estimator under first order stochastic dominance by \citet{Moesching2020} is chosen as competitor. This comparison is of interest since in situations where a variable $Y$ depends monotonically on $X$ one might want to impose first order stochastic dominance as a restriction, but it is sometimes not fully clear if the monotone relationship is strong enough such that all conditional exceedance probabilities $1-F_x(y)$ are increasing in $x$ for all $y$, or if there might be crossings of the CDFs.

\begin{figure}[h]
	\centering
	\caption{Relative improvement in $\mathrm{L}_1$ distance and mean absolute error of quantile estimates of the $\icv$- and $\icx$-order constrained estimator compared to the estimator under first order stochastic dominance, for $n = 500$. The solid lines show the improvement when the estimators are computed on the full sample, and the dashed lines for a subagging variant with $50$ subamples of size $250$. \label{fig:sim_continuous}}
	\bigskip
	\includegraphics[width = 0.9\textwidth]{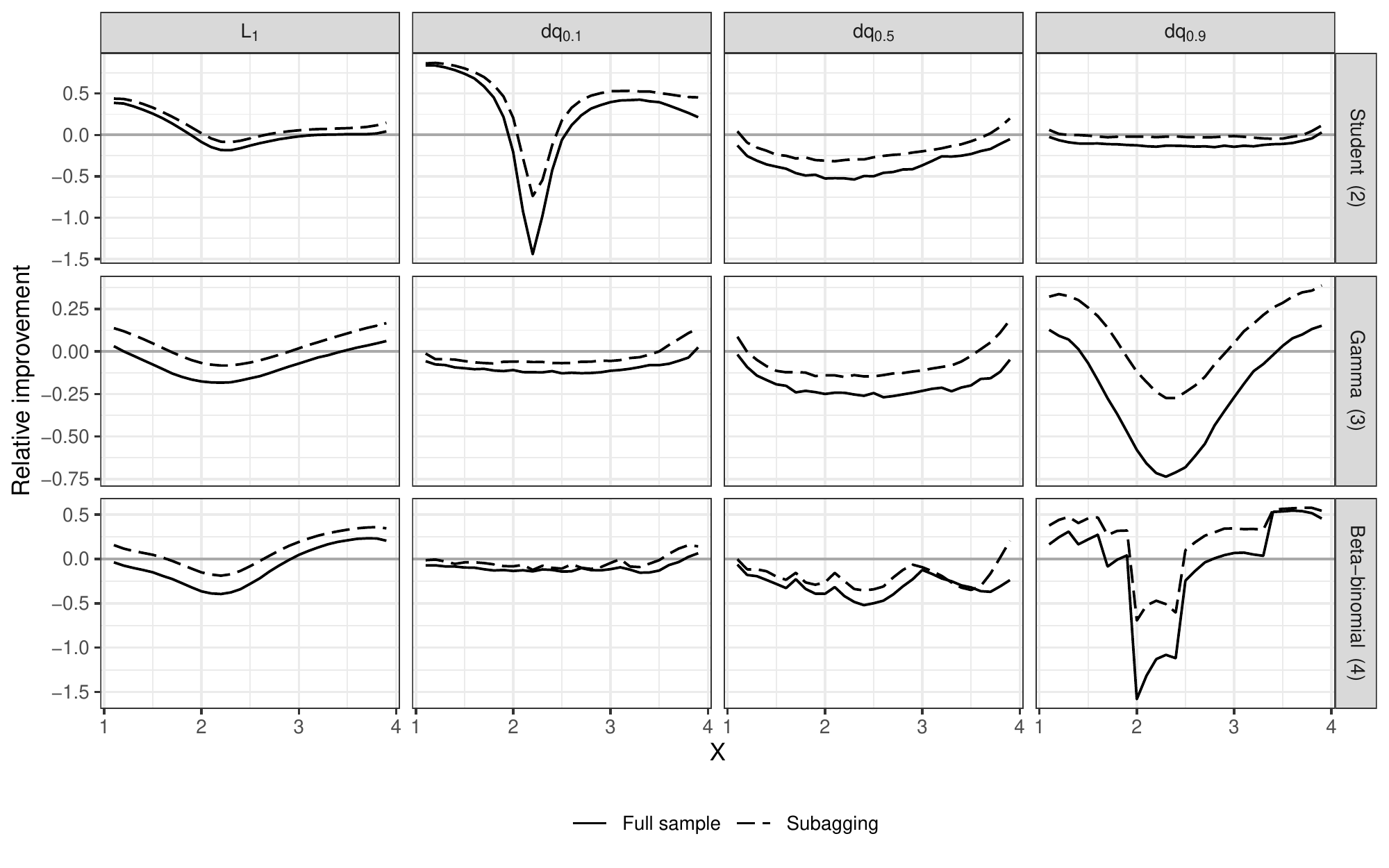}
\end{figure}

As Figure \ref{fig:sim_illustration} (a) shows, for the simulations \eqref{eq:sim_gamma} and \eqref{eq:sim_beta} the conditional quantile curves up to the seventh decile are all increasing in the covariate $X$, and so are the conditional quantile curves above the third decile in \eqref{eq:sim_student}. Therefore, although first order stochastic dominance is violated, it serves as a reasonable approximation in these problems. Figure \ref{fig:sim_continuous} shows the relative performance of the estimators for $n = 500$. The estimator by \citet{Moesching2020} achieves a lower absolute error for the median, for the $0.1$-quantile in \eqref{eq:sim_gamma} and \eqref{eq:sim_beta}, and for the $0.9$-quantile in \eqref{eq:sim_student}, uniformly over all values of $X$. This is due to the fact that the corresponding quantile curves are monotone and estimation under this correct constraint is more efficient than with the weaker $\icv$- and $\icx$-constraints. The picture is different for the low quantiles in \eqref{eq:sim_student} and the high quantiles in \eqref{eq:sim_gamma} and \eqref{eq:sim_beta}, where the conditional quantile curves are antitonic and the best isotonic approximation is constant, which generally provides a poor fit. Figure \ref{fig:sim_continuous} also compares the errors of subagging variants of the estimators; see Figure \ref{fig:sim_illustration} (b) for an illustration of the estimated quantile curves in the Gamma example. For both estimators, $50$ random subsamples of size $250 = n/2$ are drawn from the data, and the conditional CDFs from each fit to the subsamples are averaged pointwise. It can be seen that the $\icv$- and $\icx$-order constrained estimators benefit more from subagging than the estimator with first order stochastic dominance. A comparison of different subagging variants and results for other sample sizes are given in the Appendix \ref{app:sim}.

There are many other methods for estimating conditional distributions than the shape-constrained regression methods discussed so far in this article, such as models based on parametric families \citep{Rigby2005}, nonparametric kernel methods \citep{Li2008}, or quantile regression \citep{Koenker2005}, only to name a few. In general, the advantage of shape-constrained estimators is that they are free from tuning parameters and automatically adapt to the (unknown) smoothness of the underlying functions which are to be estimated, but other estimation methods can achieve a smaller estimation error when their assumptions are satisfied. Indeed, in a simulation and case study, \citet[Sections 4 and 5, Supplement S4]{Henzi2021c} have found that estimators with first order stochastic dominance constraints are often not superior to competitors in terms of $L_2$ estimation error. This can be expected also for the estimators proposed in this article, which do not generally outperform the estimator under first order stochastic dominance.

\section{Case study} \label{sec:application}
It is well known that in the the evaluation of point forecasts, a wrongly specified loss function, such as the absolute error for comparing mean forecasts, may lead to counterintuitive results and distorted forecast rankings \citep{Gneiting2011}. This causes problems in the interpretation of economic surveys, where respondents are often asked to issue point predictions for future quantities, but it is unspecified what functional of their (hypothetical) predictive distribution is meant. As a remedy, various tests of forecast rationality, or forecast calibration, have been proposed in the literature. A recent contribution is by \citet{Dimitriadis2019}, who develop tests for the hypothesis that a given point forecast is the mean, median, or mode functional, or a convex combination of the three. The case study in this section demonstrates that the estimation of conditional distributions can complement such tests to gain additional information for the interpretation of point forecasts.

If $X$ denotes a point forecast and $Y$ the observation, the hypothesis of forecast rationality with respect to a functional $T$ can be defined as $X = T[\mathcal{L}(Y \mid X)]$, where $\mathcal{L}(Y \mid X)$ denotes the conditional law of $Y$ given the forecast $X$. This formulation is a special but important case of equation (2.1) in \citet{Dimitriadis2019}, which allows including additional information available to the forecaster for conditioning. If $T$ is the mean functional, then forecast rationality is equivalent to the moment condition $\mathbb{E}(Y - X \mid X) = 0$. For the median, the corresponding condition is $\mathbb{E}(\one\{Y \geq X\}\mid X) = 0.5$, provided that $\mathcal{L}(Y \mid X)$ is a continuous distribution. Based on such conditions, \citet{Dimitriadis2019} developed asymptotic tests for forecast rationality. 

Distributional regression provides a different, more qualitative approach to this problem. If the conditional distributions $\mathcal{L}(Y \mid X = x)$ were known, one could easily derive the functional of interest $T(x) = T[\mathcal{L}(Y \mid X = x)]$ and detect violations of forecast rationality by directly comparing $T(x)$ and $x$. Estimators with stochastic order restrictions allow to mimic this ideal situation, without having to impose restrictive or implausible assumptions on the conditional distributions. For sufficiently precise point forecasts $X$, one would expect that the actual observation $Y$ tends to attain higher values as $X$ increases. Moreover, estimating $\mathcal{L}(Y \mid X)$ under stochastic order constraints only requires the ranks of the forecasts $X_1, \dots, X_n$ in a sample, but not their values. This makes a comparison of $X_i$ and $T(X_i)$ sensible, because $X_1, \dots, X_n$ themselves have not been provided to the model. Other estimation methods, such as kernel regression \citep{Li2008}, generally do not have this property.

To illustrate the approach, consider the data example from Section 5.1 of \citet{Dimitriadis2019}. In the Labor Market Survey by the Federal Reserve Bank of New York, respondents are asked three times per year to report their annualized income in four months. The sample analysed here ranges from March 2015 to November 2019. Some respondents participate in several rounds of the survey, and only the first round is included for those individuals which occur several times to obtain independent observations. Additionally, like in \citet{Dimitriadis2019}, observations with very high or low expectations or income (above 300'000 or below 1000, 4.0\% of the sample; an upper bound of 1 million was used in \citet{Dimitriadis2019}) are removed since the data is very sparse and uninformative for such values, as are cases when the ratio of expectation and income or the inverse ratio is between $9$ and $13$ (27 instances), which might be due to misplaced decimal points or erroneously reporting monthly instead of annualized income. The remaining sample consists of 3161 observations. The survey of consumer expectations (SCE; \textcopyright \ 2013-2020 Federal Reserve Bank of New York) data for this case study are available without charge at \url{https://www.newyorkfed.org/microeconomics/sce}, and may be used subject to license terms posted there. The New York Fed disclaims any responsibility for the analysis and interpretation of Survey of Consumer Expectations data in this article.

\begin{figure}
	\centering
	\caption{(a) Expected and realized income in the case study. (b) ECDF of the realized income for binned expectations. The boundaries of the bins are the $0.1$ to $0.4$-quantile of the income expectations. (c) Estimated quantile curves (levels $0.1, 0.3, 0.5, 0.7, 0.9$). (d) Mean, median and mode functional computed from the estimated conditional distributions (for expectations and incomes below 200'000).  \label{fig:case_study}}
	\bigskip
	\includegraphics[width = 0.9\textwidth]{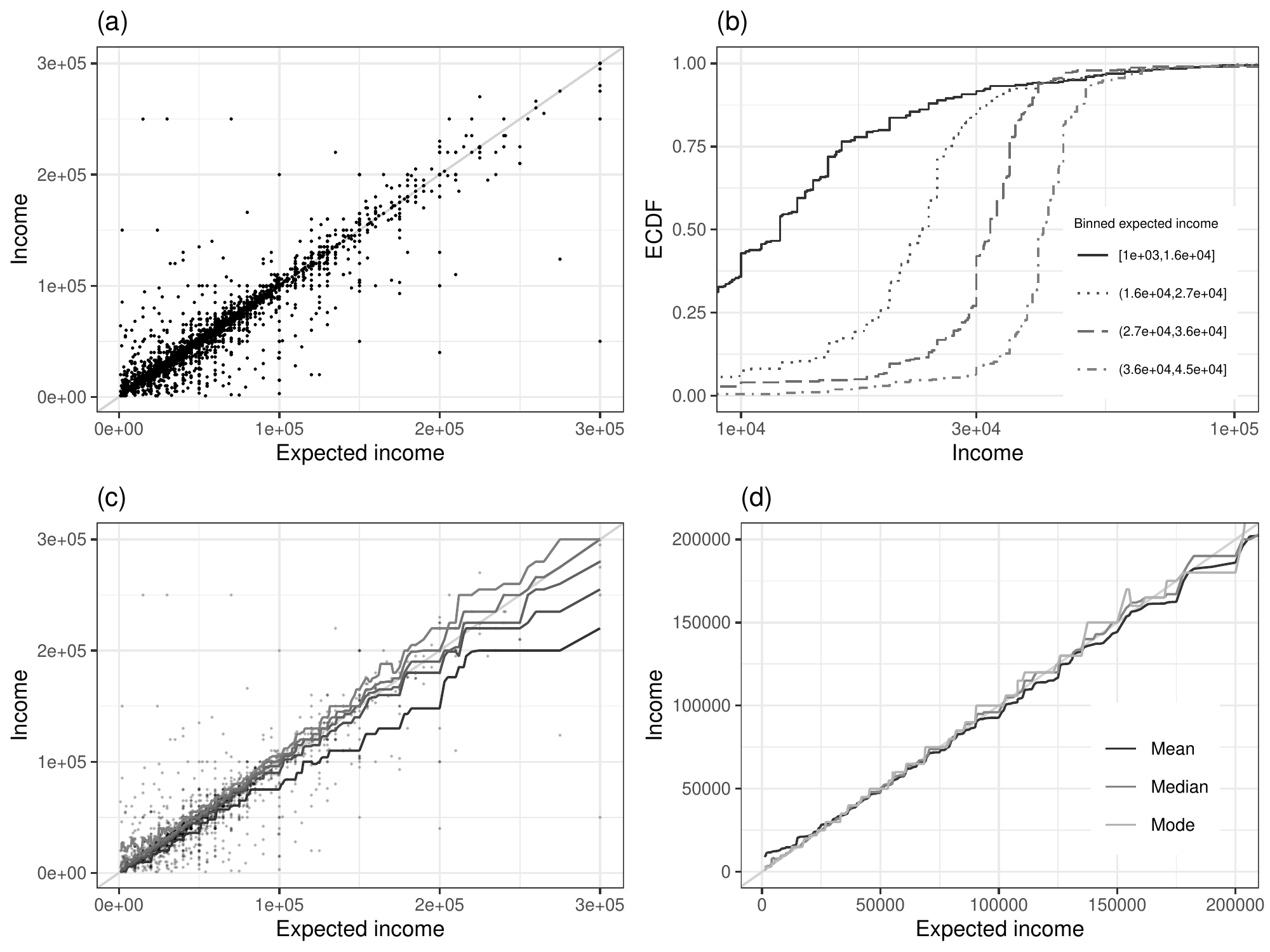}
\end{figure}

Panels (a) and (b) of Figure \ref{fig:case_study} illustrate the joint distribution of the income expectations and realizations. There is a strong monotone relationship, with a Pearson correlation of $0.92$, but interestingly, for the lowest deciles of the income expectations, the corresponding conditional CDFs of the observed income cross in the upper tail, indicating a violation of first order stochastic dominance. The intuition behind these results is that individuals are generally able to predict their income well, which explains the strong monotone relationship, but low income expectations are sometimes overly pessimistic. \citet{Rozsypal2017} have found with different data that people with lower income tend to have pessimistic expectations, both conditionally on potential confounder covariates and unconditionally. The increasing concave order can accomodate this situation, as it allows that the conditional CDFs cross in the upper tail. To estimate the conditional distributions, a subagging version of the $\icv$-order restricted estimator with $50$ subsamples of half of the total sample size is applied. From the estimated distributions, the mean, median, and mode functional are then computed, with the mode taken as the location of the largest jump of the conditional CDFs, which are piecewise constant stepfunctions. Panels (c) and (d) of Figure \ref{fig:case_study} display estimated quantile curves and the three functionals depending on the income expectation.

For the mean functional, the forecast rationality test of \citet{Dimitriadis2019} yields a p-value of $1.7\cdot 10^{-12}$, computed with the \textsf{R} package {\tt fcrat} available on \url{https://github.com/Schmidtpk/fcrat}. As can be seen in Figure \ref{fig:case_study}, the conditional mean curve lies above the bisector for expectations below 25'000, and below the bisector when the expectation exceeds 75'000, so there is indeed a systematic deviation of the income expectation from the estimated mean. For the median and the mode functional, the p-values of the rationality test are $4.5\cdot 10^{-8}$ and $0.93$, respectively. This huge difference in the p-values is in contrast to the curves in Figure \ref{fig:case_study} (d), where the expected income does not seem to deviate systematically from either functional. A simulation reveals that for data where the outcome $Y$ may be exactly equal to $X$, the p-value for the median should indeed be interpreted with care. By taking the estimated medians as new income expectation and simulating new observations from the estimated conditional distributions, one obtains new data sets where the income expectation equals the median of the underlying distribution by construction. Over 10'000 simulations, the rejection rate for the median rationality test is $0.03$, $0.11$ and $0.19$ at the levels $0.01$, $0.05$, and $0.10$ -- the test is anticonservative. The reason for the non-validity of the median rationality test is likely to be the discreteness in the data: The realized incomes only take $526$ distinct values with a sample size of $n = 3161$, and in $22\%$ of the cases the income expectation is exactly equal to the realized income. Hence the condition $\mathbb{E}(\one\{Y \geq X\}\mid X) = 0.5$ may be violated even if $X$ is equal to the conditional median due a point mass of the conditional distributions at the expected income $X$.

In conclusion, the $\icv$-constrained estimator suggests that both median and mode could rationalize the income expectations, and it confirms that the income expectations should not be interpreted as a mean forecast.

\section{Discussion}
The estimators proposed in this article may extended to more complex settings than univariate covariates $X$.
One avenue for future work is to consider partially ordered instead of real-valued covariates. A partial order relation $\preceq$ on a space $\mathcal{X}$ satisfies the same properties as the usual order of real numbers (transitivity, reflexivity, antisymmetry), but not all elements of $\mathcal{X}$ need to be comparable; an example is the componentwise order of vectors on $\mathbb{R}^p$. To construct estimators in this setting, it suffices to slightly modify the definition
\[
[\tilde{M}_{X_1}(y), \dots, \tilde{M}_{X_n}(y)] \ = \argmin_{\eta \in \mathbb{R}^n: \, \eta_i \geq \eta_j \text{ if } X_i \preceq X_j} \sum_{i = 1}^n [\eta_i - (y - Y_i)_+]^2
\]
from Section \ref{sec:estimation}, with the only difference that $X_i, X_j \in \mathbb{R}^p$ in the $\argmin$ are now compared with respect to the partial order $\preceq$. A similar min-max formula as in \eqref{eq:minmax} also applies in this case (see \citet{Barlow1972}), so that Proposition \ref{prop:welldefined} continues to hold. The computation of $[\tilde{M}_{X_1}(y), \dots, \tilde{M}_{X_n}(y)]$ is still feasible since it is a quadratic minimization problem with linear constraints, for which most statistical software programs provide efficient algorithms. While consistency with partially ordered covariates has been proved for estimation under first order stochastic dominance constraints \citep{Henzi2021c}, it remains an open problem for the orders considered in this article.

The generalization to partially ordered covariates is of interest as it allows to construct nonparametric distributional regression models with shape and scale parameters. In the spirit of \citet{Henzi2021c}, assume that we are interested in estimating the conditional distribution of a variable $Y \in \mathbb{R}$ given a collection of point forecasts $s \in \mathbb{R}^p$, say, predictions from a survey of experts or from different numerical models. This conditional distribution provides a corrected, re-calibrated version of the forecasts. An effective approach is to model $\mathcal{L}(Y \mid S = s)$ with Gaussian distributions $\mathcal{N}(a + b\bar{s}, c + d \sigma(s)^2)$, where $\bar{s}$ and $\sigma(s)^2$ are the mean and variance of the point forecasts $s$, respectively, and the parameters are estimated on training data of past forecasts and observations. \citet{Gneiting2005} apply this approach in weather forecasting, and \citet{Gneiting2010} a slightly more sophisticated model in inflation prediction. When $b, d > 0$, which usually the case or even imposed \citep{Gneiting2005}, the distributions $\mathcal{N}(a + b\bar{s}, c + d \sigma(s)^2)$ are increasing in the $\icx$-order when the vector $x = (\bar{s}, \sigma(s)) \in \mathbb{R}^2$ increases componentwise. Hence, the estimator with $\icx$-order constraints and covariate vector $(\bar{s}, \sigma(s))$ could provide a flexible nonparametric alternative to such a parametric location-scale model.

A further potential extension are distributional single index models in the spirit of \citet{Henzi2021a}, where the covariate $X$ itself is derived from a higher dimensional covariate $Z$ with some dimension reducing function $\theta$, such as $X = \theta(Z) = \alpha^{\top} Z$ for $\alpha, Z \in \mathbb{R}^p$. \citet{Henzi2021a} have shown that when the conditional CDFs $\mathbb{P}(Y \leq y \mid Z)$ of $Y$ only depend on $Z$ via $\theta(Z)$, and when these distributions are increasing in first order stochastic dominance as $\theta(Z)$ increases, then the combination of a consistent estimator $\hat{\theta}_n$ for $\theta$ and a shape constrained estimator applied to $(\hat{\theta}_n(X_i), Y_i)$, $i = 1, \dots, n$, may again yield a consistent estimator of the conditional CDFs. It is an open question whether similar results hold for the increasing concave and convex order.

\section*{Acknowledgements}
This work was supported by the Swiss National Science Foundation. The author is grateful to Johanna Ziegel and Timo Dimitriadis for helpful comments and discussions.

\bibliographystyle{apalike}
\bibliography{isoicxicv_biblio_01.bib}

\appendix

\section{Greatest convex minorants} \label{app:gcm}
Let $I \subseteq \mathbb{R}$ be an interval and $f: I \rightarrow \mathbb{R}$ a function. The greatest convex minorant of $f$ is the pointwise greatest convex function $g$ such that $g(x) \leq f(x)$ for all $x \in I$. It exists if and only if $f$ can be bounded from below by an affine linear function, and if the greatest convex minorant exists, it is unique since the pointwise supremum of convex functions is again convex. By the same reason, if $f_1$ and $f_2$ are functions with greatest convex minorants $g_1$ and $g_2$, then $f_1(x) \geq f_2(x)$ for all $x$ implies that also $g_1 \geq g_2$.

A standard result about isotonic regression \citep[see e.g.][Theorem 1.2.1]{Robertson1988} states that the isotonic regression of $z_1, \dots, z_r$ with weights $w_1, \dots, w_r > 0$, that is, the minimizer of $\sum_{i = 1}^r w_i(\theta_i - z_i)^2$ over all $\theta_1 \leq \dots \leq \theta_r$, equals the left-hand slope of the greatest convex minorant to the function that results from linearly interpolating
\[
(0, 0), \ \left(\sum_{i = 1}^k w_i, \sum_{i = 1}^k w_k z_k \right), \ k = 1, \dots, r.
\]
This result allows to describe right-hand slope of the greatest convex minorant of any piecewise linear function with finitely many knots.

\begin{lem} \label{lem:gcm}
	Let $f: [t_1, t_k] \rightarrow \mathbb{R}$ be piecewise linear with knots at $t_1 < \dots < t_k$ and let $g$ be its greatest convex minorant. Then the right-hand slope of $g$ at $t_1, \dots, t_{k - 1}$ is given by the
	isotonic regression of $[f(t_{i + 1}) - f(t_i)]/[t_{i + 1} - t_i]$ with weights $t_{i + 1} - t_i$, $i = 1, \dots, k-1$.
\end{lem}

The following lemma is known as Marshall's Inequality.
\begin{lem} \label{lem:marshall}
	Let $\mathcal{I} \subseteq \mathbb{R}$ be an interval and $f: \mathcal{I} \rightarrow \mathbb{R}$ a function, and let $g$ be the greatest convex minorant of $f$ and $h: \mathcal{I} \rightarrow \mathbb{R}$ any convex function. Assume that $\|f - h\|_{\infty} < \infty$, where $\|\cdot\|_{\infty}$ is the usual supremum norm of functions. Then, $\|g - h\|_{\infty} \leq \|f - h\|_{\infty}$.
\end{lem}
\begin{proof}
	Let $\varepsilon = \|f - h\|_{\infty}$. The function $\tilde{h}(x) = h(x) - \varepsilon$ is convex and satisfies $f(x) \geq \tilde{h}(x)$ for all $x \in \mathcal{I}$ by definition of $\varepsilon$. This and the definition of $g$ imply that $f(x) \geq g(x) \geq h(x) - \varepsilon$ for all $x \in \mathcal{I}$. 
	Since also $f(x) - h(x) \leq \varepsilon$ by the definition of $\varepsilon$, this yields
	\[
	-\varepsilon \leq g(x) - h(x) \leq f(x) - h(x) \leq \epsilon,
	\]
	and so $\|g - h\|_{\infty} \leq \varepsilon = \|f - h\|_{\infty}$.
\end{proof}

\pagebreak
\section{Proofs of theoretical results} \label{app:proofs}

Some proofs rely on properties of greatest convex minorants, which are stated in Section \ref{app:gcm}.

\begin{proof}[Proof of Proposition 2.1]
	Formula (1) in the article shows that $\tilde{M}_{x_i}(y)$ is decreasing in $i$ and increasing in $y$ when the respective other argument is fixed, and
	\begin{equation} \label{eq:lowhigh}
		\tilde{M}_{x_i}(y) = 0, \ y \leq y_1, \quad \tilde{M}_{x_i}(y_m + t) = \tilde{M}_{x_i}(y_m) + t, \ t > 0. 
	\end{equation}
	In particular, it follows that the greatest convex minorant $\hat{M}_{x_i}$ of $\tilde{M}_{x_i}$ exists. For $k, j \in \{1, \dots, d\}$ with $k \leq j$, the functions $y \mapsto \sum_{s = k}^j w_s h_s(y) / (\sum_{s = k}^{j} w_s)$
	are piecewise linear with finitely many knots, a property which is preserved when taking pointwise maxima and minima of finitely many functions. Therefore, the $\tilde{M}_{x_i}$ are also piecewise linear. For any $i \in \{1, \dots, d\}$, $y \in \mathbb{R}$ and $t > 0$,
	\begin{align*}
		0 \leq \tilde{M}_{x_i}(y + t) & = \min_{k = 1, \dots, i} \max_{j = k, \dots, d}  \frac{1}{\sum_{s = k}^{j} w_s} \sum_{s = k}^j w_s h_s(y + t) \\
		& \leq \min_{k = 1, \dots, i} \max_{j = k, \dots, d}  \frac{1}{\sum_{s = k}^{j} w_s} \sum_{s = k}^j w_s [h_s(y) + t] = \tilde{M}_{x_i}(y) + t,
	\end{align*}
	so $0 \leq [\tilde{M}_{x_i}(y + t) - \tilde{M}_{x_i}(y)] / t$, and hence $\hat{M}_{x_i}(y)$ is increasing in $y$. Lemma \ref{lem:gcm} and \eqref{eq:lowhigh} together with the inequality $[\tilde{M}_{x_i}(y + t) - \tilde{M}_{x_i}(y)] / t \leq 1$ imply that $\hat{F}_{x_i} \in [0, 1]$ with $\hat{F}_{x_i}(y) = 0$ for $y < y_1$ and $\hat{F}_{x_i}(y) = 1$ for $y \geq y_m$, and $\hat{F}_{x_i}$ is continuous from the right and increasing because is is the right-hand derivative of a convex function. Finally, $\hat{M}_{x_i}(y)$ is decreasing in $i$ because $\tilde{M}_{x_i}(y)$ is pointwise decreasing in $i$ for all $y$; see Section \ref{app:gcm}.
\end{proof}

\begin{proof}[Proof of Proposition 3.1]
	The proof is similar to the proof of Corollary 1 in \citet{Duembgen2004}. With $(c_n)_{n \in \mathbb{N}}$ from (A), define $A_n = \{\sup_{y \in \mathbb{R}, \, x \in I_n} |\tilde{M}_{n; x}(y) - M_x(y)| < c_n \}$.	Then $\lim_{n \rightarrow \infty} \mathbb{P}(A_n) = 1$, and in the following derivations, assume that the inequality in $A_n$ holds. In case (i), let $v_n = c_n^{1/(1 + \beta)}$. For $x \in I_n$, by convexity of $\hat{M}_x(\cdot)$,
	\[
	\frac{\hat{M}_{n;  x}(y) - \hat{M}_{n;  x}(y - v_n)}{v_n} \leq \hat{F}_{n;  x}(y) \leq \frac{\hat{M}_{n;  x}(y + v_n) - \hat{M}_{n;  x}(y)}{v_n},
	\]
	and the same property holds for $F_x$ and $M_x$ instead of $\hat{F}_{n;  x}$ and $\hat{M}_{n;  x}$. The function $M_x(\cdot)$ is convex, so due to Lemma \ref{lem:marshall},
	\[
	\sup_{y \in \mathbb{R}}|\hat{M}_{n;x}(y) - M_x(y)| \leq \sup_{y \in \mathbb{R}}|\tilde{M}_{n;x}(y) - M_x(y)|.
	\]
	Combining these facts yields, for any $y \in J_n$,
	\begin{align*}
		\hat{F}_{n;  x}(y) & \geq \frac{\hat{M}_{n;  x}(y) - \hat{M}_{n;  x}(y-v_n)}{v_n} \\
		& \geq \frac{M_x(y) - |\hat{M}_{n;  x}(y) - M_x(y)| - M_x(y-v_n) - |\hat{M}_{n;  x}(y-v_n) - M_x(y-v_n)|}{v_n} \\
		& \geq F_x(y-v_n) - 2c_n/v_n \\
		& \geq F_x(y) - Cv_n^{\beta} -2c_n/v_n = F_x(y) - (2 + C)c_n^{\beta/(1 + \beta)},
	\end{align*}
	and similarly
	\[
	\hat{F}_{n;  x}(y) \leq \frac{\hat{M}_{n;  x}(y + v_n) - \hat{M}_{n;  x}(y)}{v_n} \leq F_x(y) + (2 + C)c_n^{\beta/(1 + \beta)}.
	\]
	Thus $|\hat{F}_{n;  x}(y) - F_x(y)| \leq (2 + C)c_n^{\beta/(1 + \beta)}$ with on $A_n$ for $x \in I_n$ and $y \in J_n$, for each $n \in \mathbb{N}$. Under (ii), for $y \in \mathbb{Z}$ and $x \in I_n$,
	\[
	\hat{F}_{n;x}(y) = \hat{M}_{n;x}(y + 1) - \hat{M}_{n; y}(y) \leq M_x(y + 1) - M_x(y) + 2 c_n = F_x(y) + 2c_n,
	\]
	and analogously $\hat{F}_{n;x}(y) \geq F_x(y) - 2c_n$, which gives $|\hat{F}_{n;x}(y) - F_x(y)| \leq 2c_n$ For $y \in \mathbb{R} \setminus \mathbb{Z}$, the same bound is valid since $F_x(y) = F_x(\lfloor y \rfloor)$ and $\hat{F}_{n; x}(y) = \hat{F}_{n; x}(\lfloor y \rfloor)$, where the latter holds if $\tilde{M}_{n;x}(y)$ and $\hat{M}_{n;x}(y)$ are only computed at $y \in \mathbb{Z}$ and interpolated linearly.
\end{proof}

The proof of Theorem 3.2 requires several auxiliary results.

\begin{prop} \label{prop:concentration}
	Let $Z_1, \dots, Z_k$ be random variables in a non-degenerate interval $[a,b] \subset \mathbb{R}$. Then there exists a universal constant $M \leq 2^{5/2}e$ such that for all $\varepsilon > 0$,
	\[
	\mathbb{P}\left(\sup_{z \in \mathbb{R}}\frac{1}{\sqrt{k}}\Big\vert \sum_{i = 1}^k (z - Z_i)_+ - \mathbb{E}[(z - Z_i)_+]  \Big\vert \geq \varepsilon \right) \leq M\exp\left(\frac{-2\varepsilon^2}{(b-a)^2} \right)
	\]
\end{prop}

\begin{proof}
	Let $F_i$ be the cumulative distribution function of $Z_i$. The assumption $F_i(z) = 0$ for $s < a$ implies that $\mathbb{E}[(z - Z_i)_+] = \int_{a}^z F_i(z) \, ds$, so
	\begin{align*}
		\frac{1}{\sqrt{k}} \Big\vert \sum_{i = 1}^k (z - Z_i)_+ - \mathbb{E}[(z - Z_i)_+] \Big\vert&= \frac{1}{\sqrt{k}} \Big\vert\sum_{i = 1}^k \int_{a}^z \one\{Z_i \leq s\} - F_i(s) \, ds\Big\vert \\
		& \leq \frac{1}{\sqrt{k}} \int_{a}^z \Big\vert\sum_{i = 1}^k \one\{Z_i \leq s\} - F_i(s)\Big\vert \, ds \\
		& \leq \frac{1}{\sqrt{k}} \int_{a}^z \sup_{u \in \mathbb{R}} \Big\vert \sum_{i = 1}^k \one\{Z_i \leq u\} - F_i(u)\Big\vert \, ds \\
		& = \frac{1}{\sqrt{k}} (b-a) \sup_{u \in \mathbb{R}} \Big\vert \sum_{i = 1}^k \one\{Z_i \leq u\} - F_i(u)\Big\vert.
	\end{align*}
	Theorem 4.6 of \cite{Moesching2020} now yields the result.
\end{proof}

For $\gamma > 0$ and $z \in \mathbb{R}$, let $t_{\gamma}(z) = \min(\max(-\gamma, z), \gamma)$. The following inequality, which follows by simple case distinctions, will be applied several times: For all $y, z \in \mathbb{R}$,
\begin{equation} \label{eq:tuncation}
	|(y - z)_+ - (y - t_{\gamma}(z))_+| \leq (\gamma+z)_- + (z-\gamma)_+,
\end{equation}
where $(x)_- = \max(0, -x)$ and $(x)_+ = \max(0, x)$ for $x \in \mathbb{R}$.

\begin{lem} \label{lem:tail}
	Let $Z$ be a random variable such that for some $z_0 > 0$ and all $z \geq z_0$,
	\[
	\mathbb{P}(|Z| \geq z) \leq \begin{cases}
		z^{-\lambda}, \ \text{ for some } \lambda > 1, \text{ or} \\
		\exp(-\lambda z), \ \text{ for some } \lambda > 0.
	\end{cases}
	\]
	Then for $\gamma \geq z_0$,
	\[
	\mathbb{E}\left(\sup_{z \in \mathbb{R}} |(z-Z)_+ - (z-t_{\gamma}(Z))_+|\right) \leq \begin{cases}
		\gamma^{-\lambda + 1 }/(\lambda - 1), \ \text{ or } \\
		\exp(-\lambda \gamma)/\lambda.
	\end{cases}
	\]
\end{lem}
\begin{proof}
	Replacing $z$ by the random variable $Z$ in \eqref{eq:tuncation} implies that for all $\gamma \geq 0$,
	\[
	\mathbb{E}\Big( \sup_{z \in \mathbb{R}} |(z - Z)_+ - (z  - t_{\gamma}(Z))_+|\Big) \leq \mathbb{E}[(\gamma+Z)_- + (Z - \gamma)_+].
	\]
	To compute the expected value in the upper bound, let $F$ denote the cumulative distribution function of $Z$. Then,
	\[
	\mathbb{E}[(\gamma+Z)_-] = \int_{-\infty}^{-\gamma} F(z) \, ds, \quad \mathbb{E}[(Z - \gamma)_+] = \int_{\gamma}^{\infty} 1 - F(z) \, ds.
	\]
	This implies
	\[
	\mathbb{E}\Big( \sup_{z \in \mathbb{R}} |(z - Z)_+ - (z  - t_{\gamma}(Z))_+|\Big) \leq \int_{\gamma}^{\infty} F(-s) + (1 - F(s)) \, ds = \int_{\gamma}^{\infty} \mathbb{P}(|Z| \geq s) \, ds.
	\]
	In the first case, for $\gamma \geq z_0$, it holds $\int_{\gamma}^{\infty} \mathbb{P}(|Z| \geq s) \, ds \leq \gamma^{-\lambda + 1 }/(\lambda - 1)$. In the second case, the upper bound is $\exp(-\lambda \gamma)/\lambda$.
\end{proof}

Proposition \ref{prop:concentration} and Lemma \ref{lem:tail} allow to derive an analogous result to Corollary 4.7 of \citet{Moesching2020}, for which some additional notation is required. For $y \in \mathbb{R}$ and $r, s \in \{1, \dots, n\}$, $r \leq s$, define $w_{rs} = s - r + 1$ and
\[
\mathbb{M}_{rs}(y) = \frac{1}{w_{rs}}\sum_{i = r}^s (y - Y_{ni})_+, \quad \bar{M}_{rs}(y) = \frac{1}{w_{rs}}  \sum_{i = r}^s \mathbb{E}[(y - Y_{ni})_+].
\]
Recall that the estimator $\tilde{M}_{n;x_i}$ has the representation
\[
\tilde{M}_{n;x_i}(y) = \min_{k = 1, \dots, i} \max_{j = k, \dots, d}  \frac{1}{\sum_{s = k}^{j} w_s} \sum_{s = k}^j w_s h_s(y),
\]
for the distinct values $x_1 < \dots < x_d$ of $X_{n1}, \dots, X_{nn}$, $w_i = \#\{j \leq n: \, X_{nj} = x_i\}$, and
\[
h_i(y) = \frac{1}{w_i}\sum_{j: \, X_j = x_i} (y - Y_{nj})_+, \ i = 1, \dots, d.
\]
For fixed $i \in \{1, \dots, d\}$, let $1 \leq r(i) \leq s(i) \leq d$ be indices such that
\[
\tilde{M}_{n;x_i}(y) = \frac{1}{\sum_{k = r(i)}^{s(i)} w_k} \sum_{k = r(i)}^{s(i)} w_k h_k(y).
\]
Assuming $X_{n1} \leq \dots \leq X_{nn}$, with $\tilde{r}(x) = \min\{j \leq n: \, X_{nj} = x_{r(i)}\}$, $\tilde{s}(x) = \max\{j \leq n: \, X_{nj} = x_{r(i)}\}$, the estimator $\tilde{M}_{n;x_i}(y)$ equals
\[
\tilde{M}_{n;x_i}(y) = \frac{1}{\tilde{s}(i) - \tilde{r}(i) + 1} \sum_{k = \tilde{r}(i)}^{\tilde{s}(i)} (y - Y_{nk})_+.
\]
This implies that
\[
\max_{1 \leq r \leq s \leq d}\left\| \frac{1}{\sum_{k = r}^{s} w_k} \sum_{k = r}^{s} w_k (h_k- M_{x_k})  \right\|_{\infty} \leq \max_{1 \leq r \leq s \leq n}\|\mathbb{M}_{rs} - \bar{M}_{rs} \|_{\infty},
\]
and an asymptotic upper bound for $\max_{1 \leq r \leq s \leq n}\|\mathbb{M}_{rs} - \bar{M}_{rs} \|_{\infty}$ is derived below.

\begin{prop} \label{prop:bound}
	Let $R_n = \max_{1 \leq r \leq s \leq n} w_{rs}^{1/2} \|M_{rs} - \bar{M}_{rs} \|_\infty$. Then for any $D > 2$,
	\[
	\lim_{n \rightarrow \infty} \mathbb{P}\left(R_n \leq D\log(n)^{1/2}\gamma_n \right) = 1,
	\]
	where
	\[
	\gamma_n = \begin{cases}
		(n\log(n))^{1/\lambda}, & \ \text{under (P)}, \\
		2\log(n)/\lambda, & \ \text{under (E)}.
	\end{cases}
	\]
\end{prop}

\begin{proof}
	For $\gamma > 0$, define
	\[
	u(\gamma) = \begin{cases}
		\gamma^{-\lambda+1} / (\lambda-1), & \ \text{under (P)}, \\
		\exp(-\lambda \gamma)/\lambda, & \ \text{under (E),}
	\end{cases} \quad
	p(\gamma) = \begin{cases}
		\gamma^{-\lambda}, & \ \text{under (P)}, \\
		\exp(-\lambda \gamma), & \ \text{under (E).}
	\end{cases}
	\]
	By Lemma \ref{lem:tail}, for any $y \in \mathbb{R}$ and $\gamma \geq y_0$,
	\[
	\frac{1}{w_{rs}}\Big\vert \sum_{i = r}^s \mathbb{E}[(y - Y_{ni})_+] - \mathbb{E}[(y - t_{\gamma}(Y_{ni}))_+] \Big\vert \leq u(\gamma).
	\]
	Also by (P) or (E) and by \eqref{eq:tuncation},
	\[
	\mathbb{P}\left(\sup_{y \in \mathbb{R}}|(y - Y_{ni})_+ - (y - t_{\gamma}(Y_{ni}))_+| > 0 \right) \leq \mathbb{P}(|Y_{ni}| \geq \gamma) \leq p(\gamma).
	\]
	This implies that the events
	\[
	B_{n} = \left\{\sup_{y \in \mathbb{R}, i = 1, \dots, n}|(y - Y_{ni})_+ - (y - t_{\gamma}(Y_{ni}))_+| = 0 \right\}
	\]
	satisfy $\mathbb{P}(B_{n}) \geq 1 - n p(\gamma)$. Let $\prescript{}{\gamma}{\mathbb{M}_{rs}}$ and $\prescript{}{\gamma}{\bar{M}_{rs}}$ be defined as $\mathbb{M}_{rs}$ and $\bar{M}_{rs}$ but with the truncated variables $t_{\gamma}(Y_{ni})$ instead of $Y_{ni}$. By the above considerations, conditional on $B_{n}$, for any $1 \leq r \leq s \leq n$,
	\begin{align*}
		\|\mathbb{M}_{rs} - \bar{M}_{rs}\|_{\infty} = \sup_{y \in \mathbb{R}} \frac{1}{w_{rs}}\Big|\sum_{i = r}^s (y - Y_{ni})_+ - \mathbb{E}[(y - Y_{ni})_+] \Big| 
		\leq \|\prescript{}{\gamma}{\mathbb{M}_{rs}} - \prescript{}{\gamma}{\bar{M}_{rs}}\|_{\infty} + u(\gamma)
	\end{align*}
	Proposition \ref{prop:concentration} implies that
	\[
	\mathbb{P}\left(\sup_{y \in \mathbb{R}} w_{rs}^{1/2}|\prescript{}{\gamma}{\mathbb{M}_{rs}(y)} - \prescript{}{\gamma}{\bar{M}_{rs}}(y)| \geq \varepsilon \right) \leq M \exp \left(\frac{-2\varepsilon^2}{(2\gamma)^2} \right).
	\]
	Replace now $\gamma$ by 
	\[
	\gamma_n = \begin{cases}
		[n\log(n)]^{1/\lambda}, & \ \text{under (P)}, \\
		2\log(n)/\lambda, & \ \text{under (E)}.
	\end{cases}
	\]
	This yields
	\[
	n\cdot p(\gamma_n) = \begin{cases}
		n[n\log(n)]^{-\lambda/\lambda} = \log(n)^{-1}, \\
		n\exp(-2\lambda\log(n)/\lambda) = n^{-1},
	\end{cases}
	\]
	and therefore $\lim_{n \rightarrow \infty} \mathbb{P}(B_n) = 1$. Also,
	\[
	n^{1/2} \cdot u(\gamma_n) = \begin{cases}
		n^{1/2}[n\log(n)]^{(1-\lambda)/\lambda} / (\lambda-1) = n^{-1/2 + 1/\lambda}\log(n)^{1/\lambda - 1}/(\lambda-1), \\
		n^{1/2} \exp(-2\lambda\log(n)/\lambda)/\lambda = n^{-3/2}/\lambda,
	\end{cases}
	\]
	which gives $\lim_{n \rightarrow \infty}n^{1/2} \cdot u_n = 0$, using $\lambda > 2$ in the first case. For $\delta > 0$, define $\varepsilon_n = 2(1 + \delta)\log(n)^{1/2}\gamma_n$. Then, for $n$ large enough such that $n^{1/2} u(\gamma_n) \leq \delta\log(n)^{1/2}\gamma_n$, and by conditioning on $B_n$,
	\begin{align*}
		\mathbb{P}(R_n \geq \varepsilon_n) & \leq \sum_{1\leq r \leq s \leq n} \mathbb{P}(w_{rs}^{1/2}\|\mathbb{M}_{rs} - \bar{M}_{rs}\|_{\infty} \geq \varepsilon_n) \\
		& \leq \sum_{1\leq r \leq s \leq n} \mathbb{P}\left(w_{rs}^{1/2}\|\prescript{}{\gamma_n}{\mathbb{M}_{rs}} - \prescript{}{\gamma_n}{\bar{M}_{rs}}\|_{\infty} +  w_{rs}^{1/2}u(\gamma_n)\geq \varepsilon_n \right) \\
		& \leq \sum_{1\leq r \leq s \leq n} \mathbb{P}\left(w_{rs}^{1/2}\|\prescript{}{\gamma_n}{\mathbb{M}_{rs}} - \prescript{}{\gamma_n}{\bar{M}_{rs}}\|_{\infty} + n^{1/2} u(\gamma_n) \geq \varepsilon_n \right) \\
		& \leq \sum_{1\leq r \leq s \leq n} \mathbb{P}\left(w_{rs}^{1/2}\|\prescript{}{\gamma_n}{\mathbb{M}_{rs}} - \prescript{}{\gamma_n}{\bar{M}_{rs}}\|_{\infty} \geq  2(1 + \delta/2)\log(n)^{1/2}\gamma_n \right) \\
		& \leq \frac{Mn(n+1)}{2}\exp\left(- \frac{8(1 + \delta/2)^2 \log(n)\gamma_n^2}{(2\gamma_n)^2} \right) \\
		& \leq \frac{M}{2} \exp(2\log(n + 1) - 2(1 + \delta/2)^2 \log(n)) \rightarrow 0, \ n \rightarrow \infty. \qedhere
	\end{align*}
\end{proof}

\begin{proof}[Proof of Theorem 3.2, discrete setting (K)]
	For $j = 1, \dots, K$, let $A_j = \{i \in \{1, \dots, n\}: \, X_{ni} = j\}$, and define $\check{M}_{n; j} = \sum_{i \in A_i } (y - Y_{ni})_+ / \#A_i$. Recall that $\tilde{M}_{n;j}(y)$ is the antitonic regression of $(X_{ni}, (y-Y_{ni})_+)$, $i = 1, \dots, n$. Corollary B of \citet[p. 42]{Robertson1988} implies that for all $y \in \mathbb{R}$,
	\[
	\max_{j = 1, \dots, K}|M_{j}(y) - \tilde{M}_{n, j}(y)| \leq \max_{j = 1, \dots, K}|M_{j}(y) - \check{M}_{n, j}(y)|
	\]
	This gives
	\[
	\max_{j = 1, \dots, K}\|M_{j} - \tilde{M}_{n, j}\|_{\infty} \leq \max_{j = 1, \dots, K}\|M_{j} - \check{M}_{n, j}\|_{\infty}.
	\]
	Assume that $X_{n1} \leq \dots \leq X_{nn}$, and define $k(j) = \max\{k \in \{1, \dots, n\}: \, X_{nk} = j\}$ for $j = 1, \dots, K$, and $k(0) = 0$. Then $\#A_j = k(j) - k(j - 1)$, and by assumption (K),
	\[
	\min_{j = 1, \dots, K} \frac{k(j) - k(j - 1)}{n} = \frac{\#A_j}{n} \geq p / 2.
	\]
	with asymptotic probability one. Since $\check{M}_{n; j}(y) = \mathbb{M}_{(k(j-1) + 1),k(j)}(y)$ and $w_{(k(j-1) + 1),k(j)} = \#A_j$, Proposition \ref{prop:bound} implies that, with asymptotic probability one for any $D > 2$ and $j = 1, \dots, K$,
	\[
	\|\check{M}_{n; j} - M_j\|_{\infty} \leq (w_{(k(j-1) + 1),k(j)})^{-1/2} R_n \leq \left(\frac{np}{2}\right)^{-1/2}R_n \leq D\gamma_n\left(\frac{2}{p}\right)^{1/2}\left(\frac{\log(n)}{n}\right)^{1/2}.
	\]
	With $D = \sqrt{8} > 2$, the upper bound equals
	\[
	c_n = \begin{dcases}
		4p^{-1/2}n^{-1/2 + 1/\lambda}\log(n)^{1/2 + 1/\lambda}, & \ \text{under (P)}, \\
		8p^{-1/2}\lambda^{-1}n^{-1/2}\log(n)^{3/2}, & \ \text{under (E)}. \qedhere
	\end{dcases}
	\]
\end{proof}

\begin{proof}[Proof of Theorem 3.2, continuous setting (C1), (C2)]
	With Proposition \ref{prop:bound}, one can apply the same strategy of proof as for Theorem 3.3 in \cite{Moesching2020}. Let $\delta_n$ be a sequence such that $\lim_{n \rightarrow \infty} \delta_n = 0$ and $\lim_{n \rightarrow \infty} n\delta_n/\log(n) = \infty$. By assumption (C1) and by the result in Section 4.3 of \citet{Moesching2020}, for all subintervals $\mathcal{I} \subseteq I$ of length at least $\delta_n$ and any $q \in (0, p)$, the inequality $\{i \leq n: X_{ni} \in \mathcal{I}\} \geq qn\delta_n$ holds with asymptotic probability one. Let $x \in I$ such that $x - \delta_n \in I$, and define
	\[
	r(x) = \min\{i \leq n: \, X_{ni} \geq x - \delta_n\}, \ j(x) = \max\{i \leq n: \, X_{ni} \leq x\}.
	\]
	By the above considerations, with asymptotic probability one, $r(x)$ and $j(x)$ are well-defined, satisfy $r(x) \leq j(x)$, $x - \delta_n \leq X_{nr(x)} \leq X_{nj(x)} \leq x$, and $\#\{j \leq n: \, X_{nj} \in [x - \delta_n, x]\} \geq qn\delta_n$. Therefore, for any $y \in \mathbb{R}$,
	\begin{align}
		\tilde{M}_{n; x}(y) - M_x(y) & \leq \tilde{M}_{n; x_{j(x)}}(y) - M_x(y) \nonumber \\
		& = \min_{k = 1, \dots, i} \max_{j = k, \dots, d}  \frac{1}{\sum_{s = k}^{j} w_s} \sum_{s = k}^j w_s h_s(y) - M_x(y) \nonumber \\
		& \leq \max_{n \geq s \geq j(x)} \mathbb{M}_{r(x)s}(y) - M_x(y) \nonumber \\
		& \leq (qn\delta_n)^{-1/2} R_n + \max_{n \geq s \geq j(x)} \bar{M}_{r(x)s}(y) - M_x(y) \nonumber \\
		& \leq (qn\delta_n)^{-1/2} R_n + M_{x_{r(x)}}(y) - M_x(y) \label{eq:upper_bnd_error} \\
		& \leq (qn\delta_n)^{-1/2} R_n + L\delta_n, \label{eq:upper_bnd_error2}
	\end{align}
	using antitonicity of $t \mapsto \tilde{M}_t(y)$ in the first line, equation (1) from the article in the second line, and antitonicity of $t \mapsto M_t(y)$ in the second-last step. An analogous argument for $M_x(y) - \tilde{M}_{n; x}(y)$ and the asymptotic bound for $R_n$ in Proposition \ref{prop:bound} yield
	\[
	|\tilde{M}_{n; x}(y) - M_{n, x}(y)| \leq (qn\delta_n)^{-1/2} \cdot D\log(n)^{1/2}\gamma_n + L\delta_n.
	\]
	for $D > 2$. The convergence rates of these two summands are balanced if $\delta_n = ({\log(n)}/{n})^{1/3} \gamma_n^{2/3}$, and for $D = \sqrt{8}$ and $q = p/2$, the upper bound equals
	\[
	c_n = \begin{cases}
		[4p^{-1/2} + L] n^{-1/3 + 2/(3\lambda)}\log(n)^{1/3 + 2/(3\lambda)}, & \ \text{under (P)}, \\
		(2/\lambda)^{2/3}[4p^{-1/2} + L]n^{-1/3}\log(n), & \ \text{under (E)}.
	\end{cases}
	\]
	If $M_x(y)$ is constant in $x \in \mathcal{I}$ for all $y$, then the difference $M_{x_{r(x)}}(y) - M_x(y)$ in \eqref{eq:upper_bnd_error} equals zero, and the term $L\delta_n$ in \eqref{eq:upper_bnd_error2} disappears. In that case, one can set $\delta_n = \log(n)^{-1}$, which again with $D = \sqrt{8}$ and $q = p/2$ yields the upper bound
	\[
	Dq^{-1/2}n^{-1/2}\log(n)^{1/2}\gamma_n = 8p^{-1/2}\lambda^{-1}\log(n)^{2} n^{-1/2}
	\]
	under (E), which is valid for all $x$ such that $x \pm \log(n)^{-1} \in I$.
	
\end{proof}

\section{Convergence rates with interpolation}

In Section 2 in the manuscript, it is suggested to estimate $\tilde{M}_x(y)$ and $\hat{M}_x(y)$ only on a finite grid $t_1, \dots, t_k$. Below is a proof that this indeed does not influence the convergence rates, provided that $t_1 = y_1$, $t_k = y_m$, and that the grid is fine enough.

\begin{proof}[Proof that convergence rates are valid under interpolation]
	Assume that (A) holds, i.e.~
	\[
	\lim_{n \rightarrow \infty} \mathbb{P} \left(\sup_{y \in J_n, x \in I_n} |\tilde{M}_{n; x}(y) - M_{x}(y)| \geq c_n \right) = 0
	\]
	for some sequences of sets $I_n, J_n \subseteq \mathbb{R}$. Let $\tilde{m}_{n;x}$ be the linear interpolation of $\tilde{M}_{n;x}$ computed on this grid. That is, for $y \in (t_i, t_{i + 1}]$, set $\tilde{m}_{n;x}(y) = \lambda \tilde{M}_x(t_i) + (1-\lambda) \tilde{M}_{n;x}(t_{i+1})$ with $\lambda = (t_{i + 1} - y)/(t_{i+1} - t_i)$, and $\tilde{m}_{n;x}(y) = 0 = \tilde{M}_{n;x}(y)$ for $y \leq y_1 = t_1$ and $\tilde{m}_{n;x}(y) = \tilde{M}_{n;x}(y_m) + (y - y_m) = \tilde{M}_{n;x}(t_k) + (y - t_k)$ for $y \geq y_m = t_k$. Then, since $M_x(\cdot)$ is Lipschitz continuous with Lipschitz constant 1,
	\begin{align*}
		& |\tilde{m}_{n;x}(y) - M_x(y)| \leq \\
		& \quad \max\left(|\tilde{M}_{n;x}(t_i) - M_x(t_i)| + |t_i - y|, |\tilde{M}_{n;x}(t_{i+1}) - M_x(t_{i+1})| + |t_{i+1} - y| \right)
	\end{align*}
	for all $y \in \mathbb{R}$. Provided that $\sup_{i = 1, \dots, k-1}|t_i - t_{i+1}| \leq c_n$, this implies $\sup_{y \in J_n}|\tilde{m}_{n;x}(y) - M_x(y)| \leq 2c_n$, so the same convergence rate applies if $\tilde{M}_x(y)$ and $\hat{M}_x(y)$ are evaluated on a sufficiently fine grid. If $Y_{n1} < \dots < Y_{nn}$ are independent and admit a density bounded away from zero on $J \supseteq J_n$, then the results of Section 4.3 in \citet{Moesching2020} imply that $\sup_{i = 1, \dots, n-1}|Y_{ni} - Y_{n(i+1)}| \leq c_n$ holds with asymptotic probability one for the $c_n$ from Theorem 3.2 in the manuscript, so it is admissible in this case to take the observed values $y_1, \dots, y_m$ as the grid.
\end{proof}

\section{Additional figures for Section 4} \label{app:sim}

Figure \ref{fig:sim_continuous_all_n} shows the same comparison as Figure 1 in the manuscript for $n = 1000$ and $n = 1500$. In Figures \ref{fig:sim_subagging_b} and \ref{fig:sim_subagging_p}, different variants of subagging are compared. Using more than $n/2$ of the total data in subsamples is generally not better than $n/2$ or less. A higher number of subsamples improves the subagging variants of the estimators, but the effect diminishes as the number of subsamples increases.

\begin{figure}
	\centering
	\caption{Relative improvement in $\mathrm{L}_1$ distance and mean absolute error of quantile estimates of the $\icv$- and $\icx$-order constrained estimator compared to the estimator under first order stochastic dominance, for $n = 1000$ and $n = 1500$. The solid lines show the improvement when the estimators are computed on the full sample, and the dashed lines for a subagging variant with $50$ subamples of size $n/2$. \label{fig:sim_continuous_all_n}}
	\bigskip
	\includegraphics[width = \textwidth]{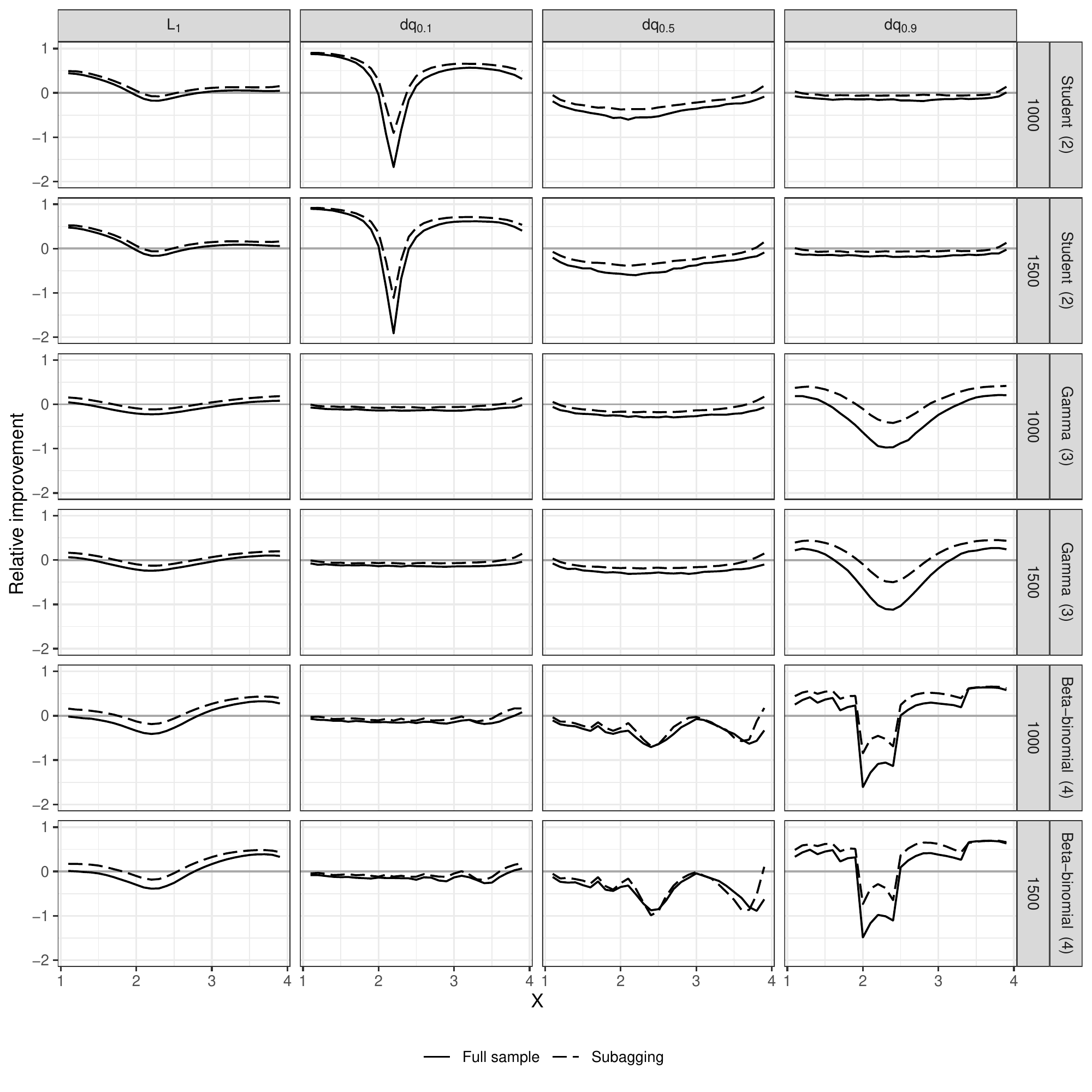}
\end{figure}

\begin{figure}
	\centering
	\caption{Relative improvement of subagging variants of the $\icv$- and $\icx$-constrained estimators (ICV/ICX) and of the estimator with first order stochastic dominance constraints (FSD) compared to the version without subagging. The sample size is $n = 1000$ fraction of data in each subsample is $n/2 = 500$.\label{fig:sim_subagging_b}}
	\bigskip
	\includegraphics[width = \textwidth]{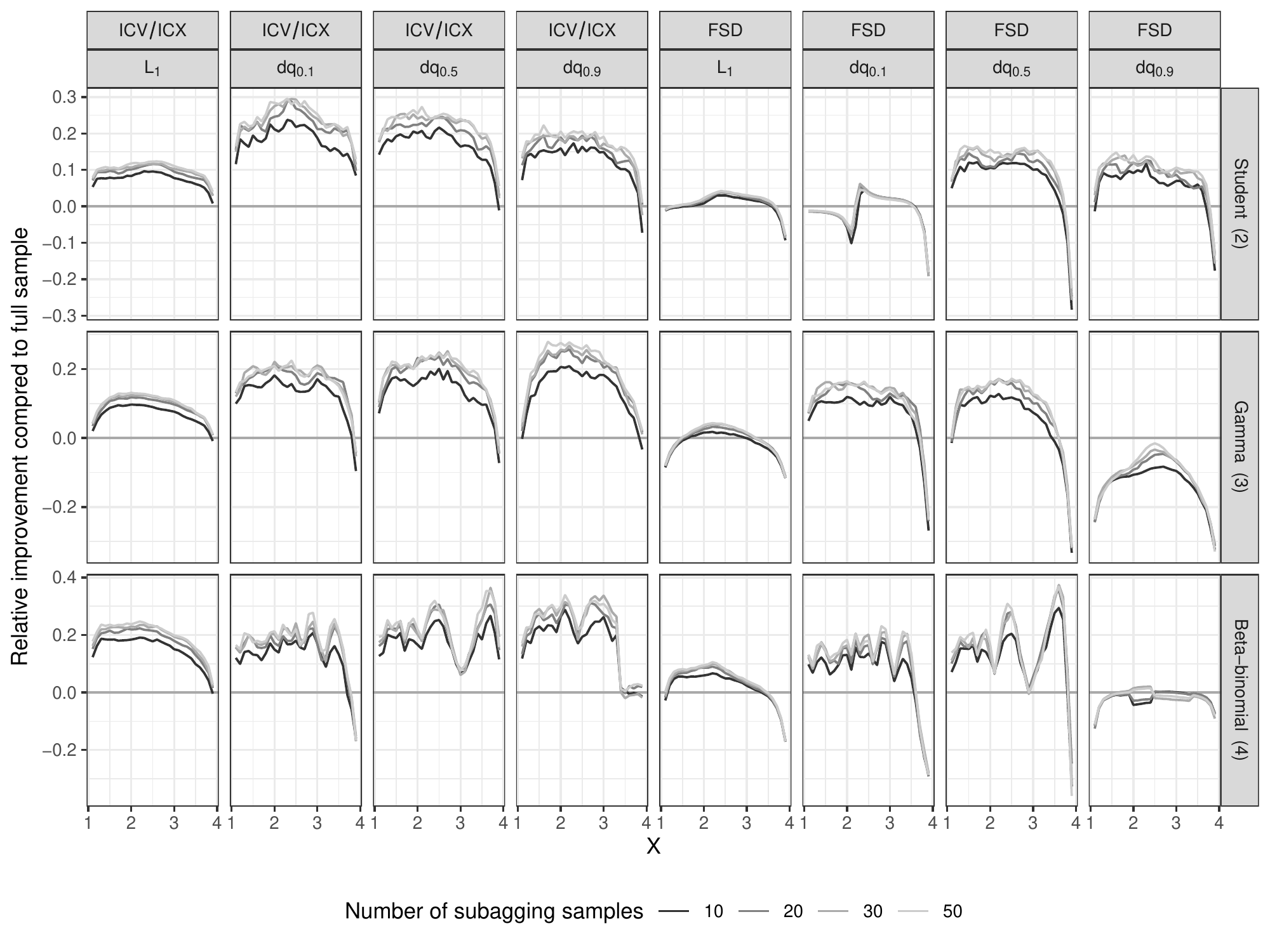}
\end{figure}

\begin{figure}
	\centering
	\caption{Relative improvement of subagging versions of $\icv$- and $\icx$-constrained estimators (ICV/ICX) and of the estimator with first order stochastic dominance constraints (FSD), compared to the variant with subsamples of size $n/2$. The sample size is $n = 1000$ and the number of subsamples is $50$ for the variants with subagging. \label{fig:sim_subagging_p}}
	\bigskip
	\includegraphics[width = \textwidth]{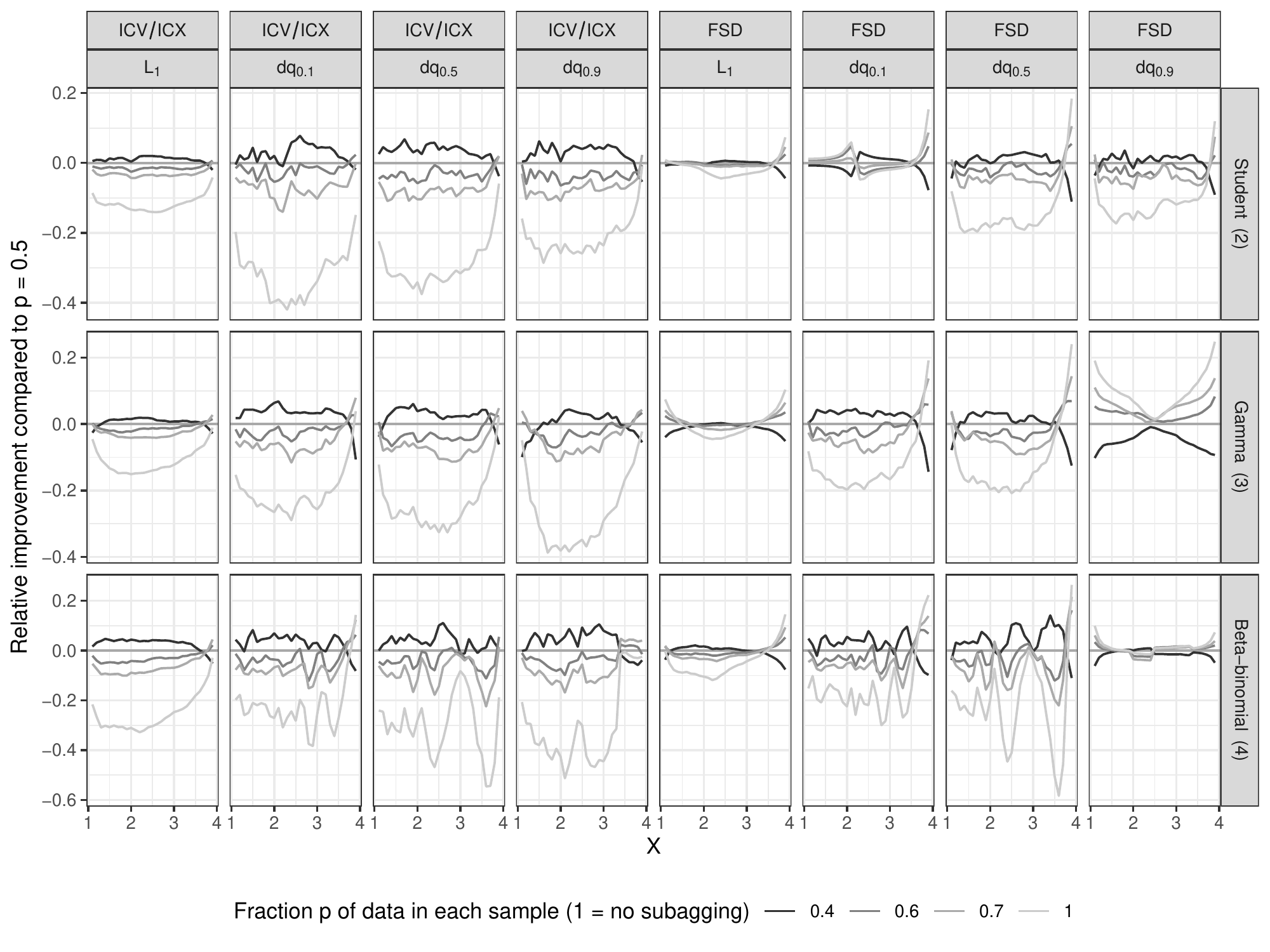}
\end{figure}

\end{document}